\pdfoutput=1
\UseRawInputEncoding
\documentclass[11pt,epsf]{article}
\usepackage[dvips]{color}
\usepackage[mathlines]{lineno}
\usepackage{bold-extra}
\usepackage{float}
\usepackage{extarrows}
\usepackage{color,latexsym,amsfonts,amssymb}
\usepackage{amsmath,amsthm}
\usepackage{mathrsfs}
\usepackage{colordvi,amsfonts,amsmath,epsfig,cite}
\usepackage{graphicx}
\usepackage{caption}
\usepackage{enumerate}
\usepackage{url}
\usepackage{subcaption}
\makeatletter
\@addtoreset{equation}{section}

\newcommand{\beq}{\begin{equation}}
\newcommand{\eeq}{\end{equation}}
\newcommand{\beqn}{\begin{equation*}}
\newcommand{\eeqn}{\end{equation*}}
\newcommand{\bed}{\begin{displaymath}}
\newcommand{\eed}{\end{displaymath}}
\newcommand{\beali}{\beq\left\{\begin{aligned}}
\newcommand{\enali}{\end{aligned}\right.\eeq}
\newcommand{\bea}{\bed\begin{array}{rl}}
\newcommand{\eea}{\end{array}\eed}

\newcommand{\olsi}[1]{\,\overline{\!{#1}}} 

\newtheorem{theorem}{Theorem}[section]
\newtheorem{lemma}[theorem]{Lemma}
\newtheorem{remark}[theorem]{Remark}

\newtheorem{definition}{Definition}[section]

\newtheorem {hypo}{Hypothesis}[section]

\begin{document}

\title{Unfolding a Hopf bifurcation in a linear reaction-diffusion equation with strongly localized impurity – existence of breathing pulses}
\author{Ji Li,\thanks{School of Mathematics and Statistics, Huazhong University of Science and Technology, Wuhan, Hubei, 430074, P.R. China,  liji@hust.edu.cn
} \and  Qing Yu,\thanks{ School of Mathematics and Statistics, Huazhong University of Science and Technology, Wuhan, Hubei, 430074, P.R. China,  yuqing@hust.edu.cn } \and Qian Zhang\thanks{Corresponding author. School of Mathematics and Statistics, Huazhong University of Science and Technology, Wuhan, Hubei, 430074, P.R. China, qian\_z@hust.edu.cn } }
\maketitle
\begin{abstract}
This paper presents a general framework to derive the weakly nonlinear stability near a Hopf bifurcation in a special class of multi-scale reaction-diffusion equations. The main focus is on how the linearity and nonlinearity of the fast variables in system influence the emergence of the breathing pulses when the slow variables are linear and the bifurcation parameter is around the Hopf bifurcation point. By applying the matching principle to the fast and slow changing quantities and using the relevant theory of singular perturbation, we obtain explicit expressions for the stationary pulses. Then, the normal form theory and the center manifold theory are applied to give  Hopf normal form expressions. Finally, one of these expressions is verified by the numerical simulation.
\end{abstract}

{\bf Key Words:}{ Pinned solutions; Hopf bifurcation; Breathing pulse; Center manifolds expansion; Normal forms.}
\\ \hspace*{\fill} \\
\textbf{Here, we consider systems of linear reaction-diffusion equations with the impact of strong, spatially localized, nonlinear impurities. From an applied perspective, it is feasible to see linear systems becoming locally nonlinear with the introduction of strongly localized impurities. Besides, these impurities exhibit the structure of Dirac delta-type function, whose role can be viewed as adding fast components to the unperturbed systems. In other words, we are studying the general singularly perturbed slow-fast reaction-diffusion equations with linear slow flow. The aim of this paper is to study localized patterns, such as breathing pulses, in the aforementioned systems with linear/nonlinear impurities. It further confirms the fact that the most general pulse destabilization scenario corresponds to a Hopf bifurcation. By applying Normal Form Theory and Center Manifold Theory, we develop a mechanism for weakly nonlinear analysis. It is worth noting that the calculations provided in this paper are much simpler than the previous analysis. Furthermore, our analysis sheds some light on the detailed unfolding of a Bogdanov–Takens bifurcation or a Dumortier–Roussarie–Sotomayor bifurcation of a localized pulse solution. }

\section{Introduction}
In this paper, we consider the unfolding of a Hopf bifurcation in a linear reaction-diffusion equation with strongly localized impurity, which was first introduced by Arjen, Heijster and Shen in \cite{DvHS} as:
\beali\label{orginal}
  \frac{\partial U_1}{\partial t}&=\frac{\partial^2 U_1}{\partial x^2}-\mu U_1+\frac{\alpha}{\varepsilon^2}I(\frac{x}{\varepsilon^2})G_1(U_1,U_2),\\
  \frac{\partial U_2}{\partial t}&=D\frac{\partial^2 U_2}{\partial x^2}-b U_1-\mu U_2+\frac{\beta}{\varepsilon^2}I(\frac{x}{\varepsilon^2})G_2(U_1,U_2),
\enali
where $(x,t)\in\mathbb{R}\times\mathbb{R^+}$, $U_{1,2}(x,t):\mathbb{R}\times\mathbb{R^+}\rightarrow\mathbb{R^\mathbb}$, $0<\varepsilon \ll1$ is a sufficiently small parameter measuring the degree of locality, $\alpha,\beta\in\mathbb{R}$ are parameters measuring the strength of the impurities, $G_1, G_2$ are sufficiently smooth functions satisfying $G_{1,2}(\mathbf{0})\neq\mathbf{0}$, and $I$ is a Dirac delta-type impurity. Here, $\mu$ is assumed to be positive to ensure the ground state 0 is stable, which has been explained in \cite{DvHS} and \cite{LSZ}. Without missing any essential idea, we assume exclusively $I(\xi)=\frac{1}{\sqrt{\pi}}e^{-\xi^2}$ all over this paper. The reasons why we study this system can be summarized into three parts. First, adding strongly localized impurities can turn the linear system into a locally nonlinear one and in reality the linearity of the models may break down under strong perturbations. Second, the system we present here avoids the difficulty in analyzing the spectrum of differential operators with nonconstant coefficients. Instead, the spectral stability problem is reduced to linear algebra by matching the changing quantities in fast and slow scales with the methods of GSPT. Third, the spatial heterogeneities eliminate the translational eigenvalue $\lambda=0$, which makes the center manifold analysis  much easier.

As one of the most general pulse destabilisation scenarios \cite{veerman2015breathing}, the mechanism of the unfolding of a Hopf bifurcation can be developed by the local analysis of the associated center manifold. With regard to the dynamic behavior near this bifurcation point, numerical simulations can reveal stable, temporally oscillating pulses. Such breathing localized pulses, also known as oscillons, are fascinating topics of research in the fields of lasers \cite{vladimirov1999numerical,Gopalakrishnan2022}, optical media \cite{gomila2006,gurevich2006breathing,turaev2012,Eslamif2023} and excitatory neural networks \cite{folias2004breathing,folias2005,bressloff2004front,bressloff2011two,bressloff2003oscillatory} over the past few decades. These oscillons are localized waves that periodically vary in shape and amplitude. They often occur in the singularly perturbed reaction-diffusion systems. They are also excitatory localized structures that can respond to external stimuli and change accordingly. Mathematical models of breathing localized pulses are often complicated. The singular perturbation theory and the center manifold reduction are required in helping analyzing their stability and bifurcations. Their existence and formation mechanisms are not fully understood and are interesting research topics in nonlinear dynamics.

Breathing pulses were initially found in one-dimensional reaction-diffusion systems by Koga and Kuramoto \cite{koga1980localized}, in which a stationary localized pattern destabilized around a Hopf bifurcation leading to a “breathing motion”. Besides, it was found that multi-spike quasi-equilibrium solutions could also undergo a Hopf bifurcation, resulting in oscillations in the spike amplitudes on an $O(1)$ time scale in \cite{chen2009,ward2003,ward2003hopf}. With respect to those oscillating fronts in \cite{folias2004breathing,bressloff2003oscillatory}, they are denoted by the terminology “breather”. In the field of optics, “breathing pulses” can be used to describe the propagation behavior of light pulses in optical fibers. In biology, “oscillons” can be employed to model the activity patterns of cells or neural networks.

Recent studies  \cite{DvHS,LSZ,CLSZ} have explored the pattern formation processes in models with linear or nonlinear structures outside the impurities based on \cite{K13,jones1995geometric}. It is worth noting that these systems can be viewed as the simplifications of some canonical singularly perturbed reaction-diffusion systems like: Gray-Scott \cite{gray1983autocatalytic,CW11,DKZ}, Gierer-Meinhardt \cite{gierer1972theory,DKP07} and Klausmeier \cite{CD18,BCD19} models, due to the existence of fast and slow scales. For these classical systems, the singularly perturbed structure induces the spatial scale separation, which gives explicit leading order expressions for the patterns under consideration. Moreover, with the help of the Hopf eigenfunctions, explicit leading order expressions for Hopf normal form coefficients can be obtained. As we have seen, \cite{veerman2015breathing} studies the emergence of the breathing pulses in a slowly nonlinear Gierer-Meinhardt system \cite{VD13}, which initially exhibits typical “subcritical" growth behavior and then a bounded temporally oscillating pulse. Different from the pattern formation in Gierer-Meinhardt system \cite{gierer1972theory,DKP07}, this additional nonlinearity has a profound impact on both the stability analysis and dynamic behavior near the Hopf bifurcation.

The numerical simulation in \cite{DvHS} also reveals such oscillating pulses. The addition of small nonlinearity $-\varepsilon^3U_1^3$ to linear $G_1$ prevents the profile from blowing up and forms a breathing pinned 1-pulse solution. This sheds some light on performing a center manifold analysis near spectral configurations of co-dimension one and higher. Hence, the unfolding of a Bogdanov-Takens bifurcation \cite{kuznetsov1998elements} or controllable chaotic pulse dynamics \cite{Eslamif2023,shil1993normal} can be a natural next step. Besides, the oscillating pulses around the Hopf bifurcation still need an analytical explanation. Hence, in this paper, we focus on the conditions under which such breathing pulses arise for general linear reaction-diffusion equations with the impact of spatial defects. To accomplish this, we need to calculate the leading order expressions for the target pulses, then apply the Hopf normal form theory to derive the normal form coefficients and finally determine the pitchfork bifurcation type. We emphasize that the main differences with \cite{veerman2015breathing} is that we avoid the analysis of the direction spanned by the translational mode in the Hopf center manifold and the computation of the inverse problems associated with the scale separated structure.

The article is structured as follows. Section \ref{2} introduces the relevant notations,  hypotheses and lemmas, mainly based on \cite{CN11}. Section \ref{3} is divided into five parts. In subsection \ref{31}, we briefly review the basic mechanism of
the existence and stability for pinned pulse solution in system (\ref{orginal}) and give explicit expressions of pulse solutions and corresponding eigenfunctions. In subsection \ref{32}, we study the local expansion of the center manifold when $G_1$ and $G_2$ are both linear. The linearities of $G_1$ and $G_2$ do not generate nonlinear remainder terms involving $U_1$ or $U_2$, which prevents breathing pulses. In subsections \ref{33} and \ref{34}, due to the difficulty in analyzing the Hopf bifurcation conditions, we examine several simple situations and discuss the cases where $G_1$ and $G_2$ are separately nonlinear. Then in subsection \ref{35}, we calculate a specific example which ever appeared in \cite{DvHS} and explain the pattern formations that emerged. Finally, in section \ref{4}, we conclude with some remarks and suggestions for future research.

\section{Notations,  Hypotheses and Lemmas}\label{2}
We briefly introduce some required notations, hypotheses and lemmas in connection with the normal form theory and the center manifold theory. As a general reference, we would like to mention \cite{CN11}.

Let $\mathcal{X}$, $\mathcal{Y}$, $\mathcal{Z}$ be (real or complex) Banach spaces such that
$\mathcal{Z}\hookrightarrow\mathcal{Y}\hookrightarrow\mathcal{X}$,
with continuous embeddings. For a parameter-dependent differential equation in $\mathcal{X}$ of the form
\beq\label{general-l}
\frac{du}{dt}=Lu+R(u,\mu),
\eeq
in which $L$ is a linear operator and $R$ is defined for $(u,\mu)$ in a neighborhood of (0,0) in $\mathcal{Z}\times\mathbb{R}^m$. Here $\mu$ is a parameter to be small.
\begin{definition}
  A linear operator $L: \mathcal{Z}\rightarrow\mathcal{X}$ is called a bounded linear operator, if $L$ is continuous. The set of bounded linear operator is denoted by $L \in \mathcal{L}(\mathcal{Z},\mathcal{X})$.
\end{definition}
\begin{hypo}\label{hypo1}
We assume that $L$ and $R$ in (\ref{general-l}) have the following properties:
\begin{enumerate}
\item[(i)] $L \in \mathcal{L}(\mathcal{Z},\mathcal{X})$;
\item[(ii)] for some $k\geq2$, there exist neighborhoods $\mathcal{V}_u\subset\mathcal{Z}$ and $\mathcal{V}_\mu\subset\mathbb{R}^m$ of 0 such that $R\in \mathcal{C}^k(\mathcal{V}_u\times\mathcal{V}_\mu,\mathcal{Y})$ and $R(0,0)=0,\ D_uR(0,0)=0$.
\end{enumerate}
\end{hypo}
\begin{hypo}\label{hypo2}
 \textbf{(Spectral decomposition)} Consider the spectrum $\sigma$ of the linear operator $L$, and write $\sigma=\sigma_+\cup\sigma_0\cup\sigma_-$,
in which $ \sigma_+=\{\lambda\in\sigma;\rm{Re}\lambda>0\},\  \sigma_0=\{\lambda\in\sigma;\rm{Re}\lambda=0\}, \ \sigma_-=\{\lambda\in\sigma;\rm{Re}\lambda<0\}$.
We assume that
\begin{enumerate}
\item[(i)] there exists a positive constant $\gamma>0$ such that
$\inf\limits_{\lambda\in\sigma_+} \rm{Re}$$\lambda >\gamma,\ \sup\limits_{\lambda\in\sigma_-} \rm{Re}$$\lambda <-\gamma$;
\item[(ii)]\label{spectral-decomposition} the set $\sigma_0$ consists of a finite number of eigenvalues with finite algebraic multiplicities.
\end{enumerate}
\end{hypo}
As a consequence of Hypothesis \ref{spectral-decomposition}(ii), we can define the spectral projection $P_0\in\mathcal{L}(\mathcal{X})$, corresponding to $\sigma_0$, by the Dunford integral formula
$P_0=\frac{1}{2\pi i}\int_{\Gamma}(\lambda \mathbb{I}-L)^{-1}d\lambda$,
where $\Gamma$ is a simple, oriented counterclockwise, Jordan curve surrounding $\sigma_0$ and lying entirely in $\{\lambda\in \mathbb{C}:|\rm{Re}\lambda|<\gamma\}$. Then we define projection $P_h:\mathcal{X}\rightarrow\mathcal{X}$ by
$P_h=\mathbb{I}-P_0$.
The spectral subspaces associated with these two projections are denoted by
$\mathcal{E}_0=\rm{Ran} $$P_0$$=\rm{Ker}$$ P_h$$\subset\mathcal{X},\ \mathcal{X}_h=\rm{Ran}$$P_h$$=\rm{Ker}$$P_0$$\subset\mathcal{X}$.
Also, we set
$\mathcal{Z}_h=P_h\mathcal{Z}\subset\mathcal{Z},\ \mathcal{Y}_h=P_h\mathcal{Y}\subset\mathcal{Y}$,
and denote by $L_0$ and $L_h$ the restrictions of $L$ to $\mathcal{E}_0$ and $\mathcal{Z}_h$. As an immediate consequence, the spectrum of $L_0$ is $\sigma_0$ and the spectrum of $L_h$ is $\sigma_h=\sigma_+\cup\sigma_-$.
\begin{hypo}\label{hypo3-1}
  \textbf{(Resolvent estimates)} Assume that $\tilde{\mathcal{X}}=\mathcal{X}\times\mathbb{R}^m$, $\tilde{\mathcal{Y}}=\mathcal{Y}\times\mathbb{R}^m$, $\tilde{\mathcal{Z}}=\mathcal{Z}\times\mathbb{R}^m$ are Hilbert spaces, and there exist positive constants $\omega_0>0$, $c>0$ and $\alpha\in[0,1)$ such that for all $\omega\in \mathbb{R}$, with $|\omega|\geq \omega_0$, we have that $i\omega$ belongs to the resolvent set of $L$ and $ ||(i\omega\mathbb{I}-\tilde{L}_h)^{-1}||_{\mathcal{L}(\mathcal{X})}\leq\frac{c}{|\omega|}$, where $\tilde{L}_h:=L_h+D_\mu P_hR(0,0)$.
\end{hypo}
\begin{lemma}\label{lemma1}
\textbf{(Parameter-dependent center manifolds)} (P.46 of \cite{CN11}) Assume that Hypotheses \ref{hypo1}, \ref{hypo2} and \ref{hypo3-1} hold. Then there exists a map $\Psi\in\mathcal{C}^k(\mathcal{E}_0\times\mathbb{R}^m,\mathcal{Z}_h)$, with
$\Psi(0,0)=0,\ D_u\Psi(0,0)=0$,
 and a neighborhood $\mathcal{O}_u\times\mathcal{O}_\mu$ of (0,0) in $\mathcal{Z}\times\mathbb{R}^m$ such that for $\mu\in\mathcal{O}_\mu$, the manifold
$\mathcal{M}_0(\mu)=\{u_0+\Psi(u_0,\mu):u_0\in\mathcal{E}_0\}$
 has the following properties:
\begin{enumerate}
\item[(i)] $\mathcal{M}_0(\mu)$ is locally invariant.
\item[(ii)] $\mathcal{M}_0(\mu)$ contains the set of bounded solutions of (\ref{general-l}) staying in $\mathcal{O}_u$ for all $t\in\mathbb{R}$.
\end{enumerate}
\end{lemma}
Specifically, when the unstable spectrum $\sigma_+$ of $L$ is empty, there is
\begin{lemma}\label{lemma2}
  \textbf{(Parameter-dependent center manifolds theorem for empty unstable spec-trum)}  (P.59 of \cite{CN11}) Assume that $\sigma_+=\emptyset$,  Hypotheses \ref{hypo1}, \ref{hypo2} and \ref{hypo3-1} hold. Then in addition to the properties in Lemma \ref{lemma1} the following holds.\\
  The local center manifold $\mathcal{M}_0(\mu)$ is locally attracting. More precisely, if $u(0)\in\mathcal{O}_u$ and the solution $u(t;u(0))$ of (\ref{general-l}) satisfies $u(t;u(0))\in \mathcal{O}_u$ for all $t>0$, then there exists $\tilde{u}\in\mathcal{M}_0 \cap \mathcal{O}_u$ and $\gamma'>0$ such that $u(t;u(0))=u(t;\tilde{u})+O(e^{-\gamma't}) \ as\ t\rightarrow\infty$.
\end{lemma}
Next, we give the hypothesis and lemmas about the normal form theorem in $\mathbb{R}^n$:
\begin{hypo}\label{hypo4}
Assume that $L$ and $R$ have the following properties:
\begin{enumerate}
\item[(i)] $L$ is a linear map in $\mathbb{R}^n$;
\item[(ii)] for some $k\geq2$, there exist neighborhoods $\mathcal{V}_u\subset\mathbb{R}^n$ and $\mathcal{V}_\mu\subset\mathbb{R}^m$ of 0 such that $R\in \mathcal{C}^k(\mathcal{V}_u\times\mathcal{V}_\mu, \mathbb{R}^n)$ and $R(0,0)=0,\ D_uR(0,0)=0$.
\end{enumerate}
\end{hypo}
\begin{lemma}\label{lemma3}
  \textbf{(Normal form for perturbed vector fields)}  (P.110 of \cite{CN11}) Assume that Hypothe-sis \ref{hypo4} holds. Then for any positive integer  $p$, $2\leq p\leq k $ , there exist neighborhoods $\mathcal{O}_1$ and $\mathcal{O}_2$ of 0 in $\mathbb{R}^n$ and $\mathbb{R}^m$, respectively, such that for any $\mu\in\mathcal{O}_2$, there is a polynomial $\Phi_\mu:\mathbb{R}^n\rightarrow\mathbb{R}^n$ of degree $p$ with the following properties:
  \begin{enumerate}
\item[(i)] The coefficients of the monomials of degree $q$ in $\Phi_\mu$ are functions of $\mu$ of class $\mathcal{C}^{k-p}$, and $\Phi_0(0)=0, \ D_u\Phi_0(0)=0$;
\item[(ii)] For $v\in\mathcal{O}_1$, the polynomial change of variable $u=v+\Phi_\mu(v)$,
transforms equation (\ref{general-l}) into the “normal form" $\frac{dv}{dt}=Lv+N_\mu(v)+\rho(v,\mu)$,
and the following properties hold:
\begin{enumerate}[a.]
    \item For any $\mu\in\mathcal{O}_2$, $N_\mu$ is a polynomial $\mathbb{R}^n\rightarrow\mathbb{R}^n$ of degree $p$
    and $N_0(0)=0,\ D_vN_0(0)=0$;
    \item The equality $N_\mu(e^{tL^*}v)=e^{tL^*}N_\mu(v)$ holds for all $(t,v)\in\mathbb{R}\times\mathbb{R}^n$ and $\mu\in\mathcal{O}_2$;
    \item The map $\rho$ belongs to $\mathcal{C}^k(\mathcal{O}_1\times\mathcal{O}_2,\mathbb{R}^n)$, and $\rho(v,\mu)=o(\|v\|^p)$ for all $\mu\in\mathcal{O}_2$.
\end{enumerate}
\end{enumerate}
\end{lemma}
\begin{lemma}\label{lemma4}  (P.11 of \cite{CN11})
Let f be a complex-valued function of class $\mathcal{C}^{k}$, $k\geq1$, defined in a neighborhood $\mathcal{O}$ of the origin in $\{(z,\bar{z}): z\in \mathbb{C}\}$, and which verifies
\beqn
f(e^{i\omega t}z, e^{-i\omega t}\bar{z})=e^{i\omega t}f(z,\bar{z})\  for\  any\  t\in\mathbb{R}\  and\  (z,\bar{z})\in\mathcal{O}.
\eeqn
Then there exists an even, complex-valued function $g$ of class $\mathcal{C}^{k-1}$ defined in a neighborhood of 0 in $\mathbb{R}$ such that $f(z,\bar{z})=zg(|z|)$. Furthermore, if $f$ is a polynomial, then $g$ is an even polynomial, $g(|z|)=\phi(|z|^2)$ for a polynomial $\phi$.
\end{lemma}
\begin{remark}
Provided that an infinite-dimensional system $\frac{du}{dt}=Lu+R(u,\mu)$ satisfies the assump-tions in center manifold Lemma \ref{lemma1}, then the reduced system satisfies Hypothesis \ref{hypo4}, so the normal form Lemma \ref{lemma2} can be employed.
\end{remark}
\section{Existence of Breathing Pulses}\label{3}
In this section, we study whether breathing pulses arise when $G_1$ and $G_2$ exhibit different linearity. First, we recall the analysis of existence and stability in \cite{CLSZ,DvHS,LSZ}. Therein, the explicit leading-order expressions of the pulse solution and the corresponding  eigenfunction are given. Then, we discuss the dynamic behavior when $G_1$ and $G_2$ are linear or nonlinear separately. In order to avoid the complexity of the discussion and the stacking of formulas, we make some simplifications. At last, a concrete example which ever appeared in \cite{DvHS} is raised, whose existence of breathing pulse can be interpreted theoretically within the framework of our previous analysis.
\subsection{The Existence and Stability of the Pinned Pulse}\label{31}
In order to obtain the leading order expressions of the pinned pulse solution and its eigenfunction, which play an important role in analyzing the normal form, we first present the basic mechanism of the existence and stability.

Consider the system (\ref{orginal}), owing to the strong spatial localization of impurities, it can be analyzed in two spatial scales, $x$ versus $\xi:=\frac{x}{\varepsilon^2}$. More specifically, the spatial domain is divided into three parts $I_s^-:=(\-\infty, -\varepsilon)$, $I_f:=[-\varepsilon, \varepsilon]$ and $I_s^+:=(\varepsilon, \infty)$. In $I_s^\pm$, the system (\ref{orginal}) can be regarded as a semi-coupled linear system, whose solution can be uniquely determined due to the exponential decay at infinity. In $I_f$, impurities are dominant, and the impact of $-\mu U_1$ and $-bU_1-\mu U_2$ are least. Hence, the leading-order accumulated changes of $\frac{dU_1}{dx}$ and $\frac{dU_2}{dx}$ can be calculated from the value of $\alpha G_1$ and $\beta G_2$. Through equalizing these accumulated changes in these two different scales, we derive the continuous solution we expect. The following stability analysis is completed in a similar way because of the same structure of the eigenvalue system and the negative essential spectrum. Actually, two different approaches to obtain the matching conditions of existence and stability are given in \cite{DvHS,LSZ,CLSZ}. One calculates the fast and slow manifolds separately, while the other gives the explicit expressions of solutions by method of variation of constant.

To system (\ref{orginal}), if
\beali\label{existence}
2\sqrt{\mu}C_1&=\alpha G_1\left(C_1, C_2+\frac{bC_1}{(-1+D)\mu}\right),\\
2\left(\sqrt{\frac{\mu}{D}}C_2+\frac{bC_1}{(-1+D)\sqrt{\mu}}\right)&=\frac{\beta}{D}G_2\left(C_1, C_2+\frac{bC_1}{(-1+D)\mu}\right)
\enali
admits non-degenerate solutions $C_1$ and $C_2$, then (\ref{orginal}) exists a nontrivial pinned pulse solution $\Gamma_p(x)=(\Gamma_{1,p}(x),\Gamma_{2,p}(x))^T=(U_{1,p}(x)+O(\varepsilon),U_{2,p}(x)+O(\varepsilon))^T$ with
\beq\label{Dno11}
U_{1,p}(x)=\begin{cases}
C_1e^{\sqrt{\mu} x}, & x \in I_s^-,\\
C_1,& x \in I_f,\\
C_1e^{-\sqrt{\mu}x}, & x \in I_s^+,
\end{cases}
\ U_{2,p}(x)=\begin{cases}
C_2e^{\sqrt{\frac{\mu}{D}}x}+\frac{b C_1}{(-1+D)\mu}e^{\sqrt{\mu}x}, & x \in I_s^-,\\
C_2+\frac{b C_1}{(-1+D)\mu}, & x \in I_f,\\
C_2e^{-\sqrt{\frac{\mu}{D}}x}+\frac{b C_1}{(-1+D)\mu}e^{-\sqrt{\mu}x}, & x \in I_s^+.
\end{cases}
\eeq
As for the spectral (and nonlinear) stability of the above pinned pulse solution, we linearize (\ref{orginal}) around $\Gamma_p$ and yield the eigenvalue problem
\beali\label{eigenfunction}
0&=\frac{d^2 P_1}{d x^2}-(\mu+\lambda) P_1+\frac{\alpha}{\varepsilon^2}I(\frac{x}{\varepsilon^2})\left(\frac{\partial G_1}{\partial U_1}(U_{1,p},U_{2,p})P_1+\frac{\partial G_1}{\partial U_2}(U_{1,p},U_{2,p})P_2\right),\\
0&=D\frac{d^2 P_2}{d x^2}-b P_1-(\mu+\lambda) P_2+\frac{\beta}{\varepsilon^2}I(\frac{x}{\varepsilon^2})\left(\frac{\partial G_2}{\partial U_1}(U_{1,p},U_{2,p})P_1+\frac{\partial G_2}{\partial U_2}(U_{1,p},U_{2,p})P_2\right).\nonumber
\enali
With $\sigma_{ess}=(-\infty, -\mu]$, it does not yield instabilities. Consequently, the stability is fully determined by the locations of eigenvalues. Hence, if
\beali\label{stability}
&\alpha\left(\frac{\partial G_1}{\partial U_1}\left(C_1,C_2+\frac{b C_1}{(-1+D)\mu}\right)C_3+\frac{\partial G_1}{\partial U_2}\left(C_1,C_2+\frac{b C_1}{(-1+D)\mu}\right)(C_4+\frac{b C_3}{(-1+D)(\mu+\lambda)})\right)\\
&=2\sqrt{\mu+\lambda}C_3,\\
&\frac{\beta}{D}\left(\frac{\partial G_2}{\partial U_1}\left(C_1,C_2+\frac{b C_1}{(-1+D)\mu}\right)C_3+\frac{\partial G_2}{\partial U_2}\left(C_1,C_2+\frac{b C_1}{(-1+D)\mu}\right)(C_4+\frac{b C_3}{(-1+D)(\mu+\lambda)})\right)\\
&=2\left(\sqrt{\frac{\mu+\lambda}{D}}C_4+\frac{b C_3}{(-1+D)\sqrt{\mu+\lambda}}\right)
\enali
admits solutions $\lambda$ ($\lambda \in \mathbb{C}/(-\infty, -\mu]$) which are all with negative real parts, then this pinned pulse solution is stable.
We denote this eigenfunction as $(P_1(x)+O(\varepsilon), P_2(x)+O(\varepsilon))^T$, whose expression can be given as
\beq\label{Dno22}
P_1(x)=\begin{cases}
C_3e^{\sqrt{\mu+\lambda}x}, & x \in I_s^-,\\
C_3, & x \in I_f,\\
C_3e^{-\sqrt{\mu+\lambda}x}, & x \in I_s^+,
\end{cases}
\ P_2(x)=\begin{cases}
C_4e^{\sqrt{\frac{\mu+\lambda}{D}}x}+\frac{b C_3}{(-1+D)(\mu+\lambda)}e^{\sqrt{\mu+\lambda}x}, & x \in I_s^-,\\
C_4+\frac{b C_3}{(-1+D)(\mu+\lambda)}, & x \in I_f,\\
C_4e^{-\sqrt{\frac{\mu+\lambda}{D}}x}+\frac{b C_3}{(-1+D)(\mu+\lambda)}e^{-\sqrt{\mu+\lambda}x}, & x \in I_s^+.
\end{cases}
\eeq
\subsection{$G_1$ and $G_2$ are Linear}\label{32}
We can cut through our analysis from one of the simplest perspectives, i.e., we assume that $G_1$ and $G_2$ are linear.

In this case, the existence condition can be verified easily because equations (\ref{existence}) are quadratic equations. For the stability condition, i.e., the solutions of equations (\ref{stability}), through analysis, it can be simplified as:
\beq\label{lambda-2}
\frac{2(\mu+\lambda)^\frac{3}{2}}{\alpha G_{13}\sqrt{D}}-\frac{G_{12}(\mu+\lambda)}{G_{13}\sqrt{D}}+\frac{b}{D+\sqrt{D}}=
\frac{\beta G_{22}(\mu+\lambda)^\frac{1}{2}}{2D}+\frac{\beta G_{22}(\mu+\lambda)}{\alpha G_{13} D}-\frac{\beta G_{12} G_{23}(\mu+\lambda)^\frac{1}{2}}{2DG_{13}},
\eeq
where $G_{ij}$ represents the $j$th coefficient of $G_i$ and $G_i=G_{i1}+G_{i2}U_1+G_{i3}U_2$ with $i=1,2$ and $j=1,2,3$. Denote $t:=\sqrt{\mu+\lambda}$, then $\rm{Re}$ $ t>0$ and  (\ref{lambda-2}) is equal to
\beq\label{lambda-3}
2t^3-(\alpha G_{12}+\frac{\beta G_{23}}{\sqrt{D}})t^2+\frac{\alpha \beta}{2\sqrt{D}}(G_{12}G_{23}-G_{13}G_{22})t+\frac{b\alpha G_{13}}{\sqrt{D}+1}=0.
\eeq
Denote $A:=-(\alpha G_{12}+\frac{\beta G_{23}}{\sqrt{D}})$, $B:=\frac{\alpha \beta}{2\sqrt{D}}(G_{12}G_{23}-G_{13}G_{22})$ and $E:=\frac{b\alpha G_{13}}{\sqrt{D}+1}$, we rewrite equation (\ref{lambda-3}) as
\beq\label{H}
H(t):=2t^3+At^2+Bt+E=0.
\eeq
For $B<0$, $H(t)$ has a minimum at $t=\frac{-A+\sqrt{A^2-6B}}{6}$. If, in addition, $E>0$ and $H(\frac{-A+\sqrt{A^2-6B}}{6})<0$, then (\ref{H}) has two real-valued positive solutions. When the value $H(\frac{-A+\sqrt{A^2-6B}}{6})$ becomes positive, these two real-valued solutions merge and become complex-valued. As a result, we can tune $\mu$ such that $\rm{Re}\lambda=0$. Hence, this pinned pulse solution can undergo a Hopf bifurcation, we denote this bifurcation parameter as $\hat{\mu}$. Next, we discuss the center normal form. We transform the system (\ref{orginal}) by $ \tilde{U}_1=U_1-\Gamma_{1,p},\ \tilde{U}_2=U_2-\Gamma_{2,p},\ \tilde{\mu}=\mu-\hat{\mu}$, to
\beq\label{general-L}
\frac{d\tilde{U}}{dt}=L\tilde{U}+R(\tilde{U},\tilde{\mu}),\nonumber
\eeq
where
\beqn
\begin{aligned}
\tilde{U}=
\begin{pmatrix}
\tilde{U}_1 \\
\tilde{U}_2
\end{pmatrix},
\ \ R(\tilde{U},\tilde{\mu})=
\begin{pmatrix}
-\tilde{\mu} \tilde{U}_1 \\
-\tilde{\mu} \tilde{U}_2
\end{pmatrix},
\end{aligned}
\eeqn
\beqn
\begin{aligned}
L=
\begin{pmatrix}
\frac{d^2}{dx^2}-\hat{\mu}+\frac{\alpha}{\varepsilon^2}I(\frac{x}{\varepsilon^2})\frac{\partial G_1}{\partial U_1}(\Gamma_{1,p},\Gamma_{2,p}) & \frac{\alpha}{\varepsilon^2}I(\frac{x}{\varepsilon^2})\frac{\partial G_1}{\partial U_2}(\Gamma_{1,p},\Gamma_{2,p}) \\
-b+\frac{\beta}{\varepsilon^2}I(\frac{x}{\varepsilon^2})\frac{\partial G_2}{\partial U_1}(\Gamma_{1,p},\Gamma_{2,p}) & D\frac{d^2}{dx^2}-\hat{\mu}+\frac{\beta}{\varepsilon^2}I(\frac{x}{\varepsilon^2})\frac{\partial G_2}{\partial U_2}(\Gamma_{1,p},\Gamma_{2,p})  \\
\end{pmatrix}.
\end{aligned}
\eeqn
 Since $G_1$ and $G_2$ are linear, actually,
\beq
L=
\begin{pmatrix}
\frac{d^2}{dx^2}-\hat{\mu}+\frac{\alpha G_{12}}{\varepsilon^2}I(\frac{x}{\varepsilon^2}) & \frac{\alpha G_{13}}{\varepsilon^2}I(\frac{x}{\varepsilon^2}) \\
-b+\frac{\beta G_{22}}{\varepsilon^2}I(\frac{x}{\varepsilon^2}) & D\frac{d^2}{dx^2}-\hat{\mu}+\frac{\beta G_{23}}{\varepsilon^2}I(\frac{x}{\varepsilon^2}) \nonumber \\
\end{pmatrix}.
\eeq
We can observe that $R$ is linear with respect to $\tilde{U}$ and $\tilde{\mu}$ respectively, which indicates that the $\Psi$ we calculated by Lemma \ref{lemma1} is linear with $\mu$-dependent coefficients. Therefore, there is no way for the center manifold to own a periodic orbit. Besides, the derivative of $R$ with respect to $\tilde{U}$ is $-\tilde{\mu}$, which decides the in(stability) of $\Gamma_p$. Here, we use $\mathcal{X}=\left(L^2(\mathbb{R}^2) \right)^2$, $\mathcal{Y}=\mathcal{Z}=\left(H^2(\mathbb{R}^2) \right)^2$. And we get the conclusions:
\begin{remark}\label{linear-rem}
  The differential equation (\ref{orginal}) with linear $G_1$ and linear $G_2$ has precisely one equili-brium $\Gamma_p(\mu)$, which is stable when $\mu>\hat{\mu}$ and unstable when $\mu<\hat{\mu}$. And, there exists no breathing phenomenon near the pinned pulse solution.
\end{remark}
\begin{remark}\label{rem1}
  When $D=1$, the expressions for (\ref{Dno11}) and (\ref{Dno22}) fail, the corresponding pulse solution and eigenfunction become
\beq\label{D1}
U_{2,p}(x)=\begin{cases}
C_2e^{\sqrt{\mu} x}+\frac{b C_1}{4\mu}e^{\sqrt{\mu} x}\left(-1+2\sqrt{\mu} x\right), & x \in I_s^-,\\
C_2-\frac{b C_1}{4\mu}, & x \in I_f,\\
C_2e^{-\sqrt{\mu} x}+\frac{b C_1}{4\mu}e^{-\sqrt{\mu} x}\left(-1-2\sqrt{\mu} x\right), & x \in I_s^+,\nonumber
\end{cases}
\eeq
and
\beq\label{D2}
P_2(x)=\begin{cases}
C_4e^{\sqrt{\mu+\lambda} x}+\frac{b C_3}{4(\mu+\lambda)}e^{\sqrt{\mu+\lambda} x}\left(-1+2\sqrt{\mu+\lambda} x\right), & x \in I_s^-,\\
C_4-\frac{b C_3}{4(\mu+\lambda)}, & x \in I_f,\\
C_4e^{-\sqrt{\mu+\lambda} x}+\frac{b C_3}{4(\mu+\lambda)}e^{-\sqrt{\mu+\lambda} x}\left(-1-2\sqrt{\mu+\lambda} x\right), & x \in I_s^+.\nonumber
\end{cases}
\eeq
Here $C_1$, $C_2$ satisfy
\beali
2\sqrt{\mu}C_1&=\alpha G_1\left(C_1, C_2-\frac{b C_1}{4\mu}\right),\\
2\left(\sqrt{\mu}C_2+\frac{b C_1}{4\sqrt{\mu}}\right)&=\beta G_2\left(C_1, C_2-\frac{b C_1}{4\mu}\right).\nonumber
\enali
$C_3$, $C_4$ satisfy
\beali
&2\sqrt{\mu+\lambda}C_3\\
&=\alpha\left(\frac{\partial G_1}{\partial U_1}\left(C_1,C_2-\frac{b C_1}{4\mu}\right)C_3+\frac{\partial G_1}{\partial U_2}\left(C_1,C_2-\frac{b C_1}{4\mu}\right)(C_4-\frac{b C_3}{4(\mu+\lambda)})\right),\\
&2\left(\sqrt{\mu+\lambda}C_4+\frac{b C_3}{4\sqrt{\mu+\lambda}}\right)\\
&=\frac{\beta}{D}\left(\frac{\partial G_2}{\partial U_1}\left(C_1,C_2-\frac{b C_1}{4\mu}\right)C_3+\frac{\partial G_2}{\partial U_2}\left(C_1,C_2-\frac{b C_1}{4\mu}\right)(C_4-\frac{b C_3}{4(\mu+\lambda)})\right),\nonumber
\enali
which yields
\beqn
\frac{2}{\alpha G_{13}}(\mu+\lambda)^\frac{3}{2}-\frac{G_{12}}{G_{13}}(\mu+\lambda)+\frac{b}{2}=
\frac{\beta G_{22}}{2}(\mu+\lambda)^\frac{1}{2}+\frac{\beta G_{23}}{\alpha G_{13} }(\mu+\lambda)-\frac{\beta G_{12} G_{23}}{2G_{13}}(\mu+\lambda)^\frac{1}{2}.
\eeqn
It equals to equation (\ref{lambda-2}) when $D=1$. Hence, we get the same conclusion as in Remark \ref{linear-rem}.
\end{remark}
\subsection{$G_1$ is Nonlinear and $G_2$ is Linear}\label{33}
Next, we examine the sub-simple situation, i.e., $G_1$ is nonlinear and $G_2$ is linear. We further denote $G_1=G_{11}+G_{12}U_1+G_{13}U_2+G_{14}U_1^2+G_{15}U_2^2+G_{16}U_1U_2+...$. In order to reveal the essence, we shift our focus on the coupled case when there is only one nonlinear term, i.e., $G_1=G_{11}+G_{15}U_2^2$. In order not to bring in cumbersome formulas and complex expressions from $G_2$, we discuss from three situations.
\subsubsection{$G_2=G_{21}$ }
First, we assume that $G_2$ is a constant function with $G_2=G_{21}$, then (\ref{orginal}) becomes
\beali\label{orginal-nonlinear-1}
  \frac{\partial U_1}{\partial t}&=\frac{\partial^2 U_1}{\partial x^2}-\mu U_1+\frac{\alpha}{\varepsilon^2}I(\frac{x}{\varepsilon^2})(G_{11}+G_{15}U_2^2),\\
  \frac{\partial U_2}{\partial t}&=D\frac{\partial^2 U_2}{\partial x^2}-b U_1-\mu U_2+\frac{\beta}{\varepsilon^2}I(\frac{x}{\varepsilon^2})(G_{21}).
\enali
As before, system (\ref{orginal-nonlinear-1}) admits the stationary pinned pulse solution like (\ref{Dno11}) if
\beali\label{cond1}
2\sqrt{\mu}C_1&=\alpha\left(G_{11} +G_{15}(C_2+\frac{b C_1}{(-1+D)\mu})^2\right),\\
\frac{\beta G_{21}}{D}&=2\left(\sqrt{\frac{\mu}{D}}C_2+\frac{bC_1}{(-1+D)\sqrt{\mu}}\right)
\enali
admits non-degenerate solutions $C_1$ and $C_2$, i.e., $\Delta=4\mu+\frac{4\alpha G_{15} b}{(1+\sqrt{D})\mu}\left(\frac{\beta G_{21}}{\sqrt{D}}-\frac{\alpha b G_{11}}{\mu(1+\sqrt{D})}\right)>0$. Here, we assume that $D\neq1$. Unless otherwise stated, we adhere to this assumption in the following.

Then, under this assumption of existence, consider solutions of the eigenvalue equation
\beali\label{orginal-nonlinear-2}
  0&=\frac{d^2 P_1}{d x^2}-(\mu+\lambda) P_1+\frac{\alpha}{\varepsilon^2}I(\frac{x}{\varepsilon^2})(2G_{15}\Gamma_{2,p}(x)P_2),\\
  0&=D\frac{d^2 P_2}{d x^2}-b P_1-(\mu+\lambda) P_2,
\enali
i.e., the solutions of
\beali\label{cond2}
2\sqrt{\mu+\lambda}C_3&=2\alpha G_{15}\left(C_2+\frac{b C_1}{(-1+D)\mu}\right)\left(C_4+\frac{b C_3}{(-1+D)(\mu+\lambda)}\right),\\
0&= 2\left(\sqrt{\frac{\mu+\lambda}{D}}C_4+\frac{b C_3}{(-1+D)\sqrt{\mu+\lambda}}\right).
\enali
We solve it to get
\beqn
\begin{aligned}
\mu+\lambda &=\left|\left(C_2+\frac{b C_1}{(-1+D)\mu}\right)\frac{-\alpha b G_{15} }{1+\sqrt{D}}\right|^{\frac{2}{3}},\\
\mu+\lambda &=\left|\left(C_2+\frac{b C_1}{(-1+D)\mu}\right)\frac{-\alpha b G_{15} }{1+\sqrt{D}}\right|^{\frac{2}{3}}(-\frac{1}{2}+\frac{\sqrt{3}}{2}i),\\
\mu+\lambda &=\left|\left(C_2+\frac{b C_1}{(-1+D)\mu}\right)\frac{-\alpha b G_{15} }{1+\sqrt{D}}\right|^{\frac{2}{3}}(-\frac{1}{2}-\frac{\sqrt{3}}{2}i),
\end{aligned}
\eeqn
where $\left|\left(C_2+\frac{b C_1}{(-1+D)\mu}\right)\frac{-\alpha b G_{15}}{1+\sqrt{D}}\right|^{\frac{2}{3}}$ is positive and real. This outcome implies that Hopf bifurcation is impossible in this case regardless of $\mu$ values.
\begin{remark}\label{th1}
  There exists no Hopf bifurcation in system (\ref{orginal-nonlinear-1}).
\end{remark}
\subsubsection{$G_2=G_{22}U_1$ }
Next, we consider the case that $G_2$ is simple linear in $U_1$, i.e.,
 \beali\label{orginal-nonlinear-3}
  \frac{\partial U_1}{\partial t}&=\frac{\partial^2 U_1}{\partial x^2}-\mu U_1+\frac{\alpha}{\varepsilon^2}I(\frac{x}{\varepsilon^2})(G_{11}+G_{15}U_2^2),\\
  \frac{\partial U_2}{\partial t}&=D\frac{\partial^2 U_2}{\partial x^2}-b U_1-\mu U_2+\frac{\beta}{\varepsilon^2}I(\frac{x}{\varepsilon^2})(G_{22}U_1).
\enali
Provided that
\begin{flalign*}
& (A1)\ \Delta=4\mu-4\alpha^2G_{15}G_{11}\left(\frac{\beta G_{22}}{2\sqrt{\mu D}}-\frac{b}{\mu(1+\sqrt{D})}\right)^2>0,&
\end{flalign*}
then system (\ref{orginal-nonlinear-3}) exists a unique pinned pulse solution.

Under this assumption, the eigenvalues can be calculated by
\beq\label{cubic}
(\mu+\lambda)^{\frac{3}{2}}=\alpha G_{15}\left(C_2+\frac{b C_1}{(-1+D)\mu}\right)\left(\frac{\beta G_{22}}{2\sqrt{D}}(\mu+\lambda)^{\frac{1}{2}}-\frac{b}{1+\sqrt{D}}\right).
\eeq
Similarly, denote $B_1:=-\frac{\alpha \beta G_{15}G_{22}}{2\sqrt{D}}\left(C_2+\frac{b C_1}{(-1+D)\mu}\right)$, $E_1:=\frac{\alpha b G_{15}}{1+\sqrt{D}}\left(C_2+\frac{b C_1}{(-1+D)\mu}\right)$, we rewrite (\ref{cubic}) as
\beq
H_1(t):=t^3+B_1t+E_1=0.\nonumber
\eeq
If $E_1>0$, $B_1<0$ and $H_1(\sqrt{\frac{-B_1}{3}})<0$, then there exist two real-valued positive solutions. These two real-valued solutions merge and become complex-valued when $H_1(\sqrt{\frac{-B_1}{3}})>0$. Hence, if
\begin{flalign*}
&(A2)\   \frac{\alpha b G_{15}}{1+\sqrt{D}}\left(C_2+\frac{b C_1}{(-1+D)\mu}\right)>0;&\\
&(A3)\  \alpha \beta G_{15}G_{22}\left(C_2+\frac{b C_1}{(-1+D)\mu}\right)>0;   \\
&(A4)\   \left(\frac{\alpha \beta G_{15}G_{22}}{3}\left(C_2+\frac{b C_1}{(-1+D)\mu}\right)\right)^\frac{3}{2}+\frac{\alpha b G_{15}}{1+\sqrt{D}}\left(C_2+\frac{b C_1}{(-1+D)\mu}\right)\\
&\ \ \ \ \ -\alpha \beta G_{15}G_{22}\left(C_2+\frac{b C_1}{(-1+D)\mu}\right)\left(\frac{\alpha \beta G_{15}G_{22}}{3}\left(C_2+\frac{b C_1}{(-1+D)\mu}\right)\right)^\frac{1}{2}>0,\ \ \ \ \ \ \ \ \ \ \ \ \ \ \ \ \ \ \ \ \ \ \
\end{flalign*}
then there exist a pair of conjugate imaginary roots. According to the Shengjin formula \cite{sj89} of the cubic equation, this pair of conjugate imaginary solutions can be calculated to be
\beq
\begin{aligned}\nonumber
t_{1,2}&=\frac{1}{6}\left(\sqrt[3]{\frac{3}{2}\left(9C_1+\sqrt{81C_1^2+12B_1^3}\right)}+\sqrt[3]{\frac{3}{2}\left(9C_1-\sqrt{81C_1^2+12B_1^3}\right)}\right)\\
&\pm\frac{\sqrt{3}}{6}\left(\sqrt[3]{\frac{3}{2}\left(9C_1+\sqrt{81C_1^2+12B_1^3}\right)}-\sqrt[3]{\frac{3}{2}\left(9C_1-\sqrt{81C_1^2+12B_1^3}\right)}\right)i.
\end{aligned}
\eeq
Denote
\beqn
\begin{aligned}
n_r:&=\frac{1}{6}\left(\sqrt[3]{\frac{3}{2}\left(9C_1+\sqrt{81C_1^2+12B_1^3}\right)}+\sqrt[3]{\frac{3}{2}\left(9C_1-\sqrt{81C_1^2+12B_1^3}\right)}\right), \\ n_i:&=\frac{\sqrt{3}}{6}\left(\sqrt[3]{\frac{3}{2}\left(9C_1+\sqrt{81C_1^2+12B_1^3}\right)}-\sqrt[3]{\frac{3}{2}\left(9C_1-\sqrt{81C_1^2+12B_1^3}\right)}\right),
\end{aligned}
\eeqn
then setting $\mu=n_r^2-n_i^2$ can give $\lambda=\pm2n_rn_ii$. This equals to the condition
\begin{flalign*}
&(A5)\ 9\mu =2\sqrt[3]{\frac{3}{2}\left(\frac{-9\alpha b G_{15}\left(C_2+\frac{b C_1}{(-1+D)\mu}\right)}{1+\sqrt{D}}+\sqrt{\Delta}\right)}\sqrt[3]{\frac{3}{2}\left(\frac{-9\alpha b G_{15}\left(C_2+\frac{b C_1}{(-1+D)\mu}\right)}{1+\sqrt{D}}-\sqrt{\Delta}\right)}& \\
&\ -\frac{1}{2}\sqrt[3]{\left(\frac{-27\alpha b G_{15}\left(C_2+\frac{b C_1}{(-1+D)\mu}\right)}{2+2\sqrt{D}}+\frac{3}{2}\sqrt{\Delta}\right)^2}-\frac{1}{2}\sqrt[3]{\left(\frac{-27\alpha b G_{15}\left(C_2+\frac{b C_1}{(-1+D)\mu}\right)}{2+2\sqrt{D}}-\frac{3}{2}\sqrt{\Delta}\right)^2}
\end{flalign*}
\begin{flalign*}
&\  admits\  positive\  roots\  \mu,\ where &\\
&\ \Delta=81\left(\left(C_2+\frac{b C_1}{(-1+D)\mu}\right)\frac{\alpha b G_{15} }{1+\sqrt{D}}\right)^2-12\left(\left(C_2+\frac{b C_1}{(-1+D)\mu}\right)\frac{\alpha \beta G_{22}G_{15}}{2\sqrt{D}}\right)^3.
\end{flalign*}
As a result, there is
\begin{theorem}\label{th3}
  Assume that (A1), (A2), (A3), (A4) and (A5) hold, then there exists a Hopf bifurcation for system (\ref{orginal-nonlinear-3}), where we denote this bifurcation parameter value by $\hat{\mu}$.
\end{theorem}
Based on this, we proceed to study the existence of breathing pulses. First, we transform the system (\ref{orginal-nonlinear-3}) and get
\beq\label{G22}
\frac{d\tilde{U}}{dt}=L\tilde{U}+R(\tilde{U},\tilde{\mu}),
\eeq
where
\beq\label{G22-1}
\begin{aligned}
\tilde{U}=
\begin{pmatrix}
\tilde{U}_1 \\
\tilde{U}_2
\end{pmatrix},
\ \ R(\tilde{U},\tilde{\mu})=
\begin{pmatrix}
-\tilde{\mu} \tilde{U}_1+\frac{\alpha G_{15}}{\varepsilon^2}I(\frac{x}{\varepsilon^2})\tilde{U}_2^2 \\
-\tilde{\mu} \tilde{U}_2
\end{pmatrix},
\ L=
\begin{pmatrix}
\frac{d^2}{dx^2}-\hat{\mu} & \frac{2\alpha G_{15}}{\varepsilon^2}I(\frac{x}{\varepsilon^2})\Gamma_{2,p} \\
-b+\frac{\beta G_{22}}{\varepsilon^2}I(\frac{x}{\varepsilon^2}) & D\frac{d^2}{dx^2}-\hat{\mu}  \\
\end{pmatrix}.
\end{aligned}
\eeq
 We still use the function space $\mathcal{X}=\left(L^2(\mathbb{R}^2) \right)^2$, $\mathcal{Y}=\mathcal{Z}=\left(H^2(\mathbb{R}^2) \right)^2$, and $L:\mathcal{Z}=\mathcal{Y}\rightarrow\mathcal{X}$. $L$ is a closed operator in $\mathcal{X}$ with domain $\mathcal{Z}$. Since the operator $L$ is real, and the fact that its point spectrum is given as $\lambda=\pm2n_rn_ii$, we know the associated spectral subspace $\mathcal{E}_0$ is two-dimensional spanned by the eigenfunctions ${P,\bar{P}}$. And $P=(P_1,P_2)^T$ satisfies
 \beali\nonumber
0&=\frac{\partial^2 P_1}{\partial x^2}-(n_r^2-n_i^2+2n_rn_ii)P_1+\frac{\alpha}{\varepsilon^2}I(\frac{x}{\varepsilon^2})(2G_{15}\Gamma_{2,p}^2P_2),\\
0&=D\frac{\partial^2 P_2}{\partial x^2}-b P_1-(n_r^2-n_i^2+2n_rn_ii)P_2+\frac{\beta}{\varepsilon^2}I(\frac{x}{\varepsilon^2})(G_{22}P_1).
\enali
According to the parameter-dependent center manifolds Lemma \ref{lemma1}, we have
\beqn
\tilde{U}=U_0+\Psi(U_0,\tilde{\mu}), \ U_0\in\mathcal{E}_0,\ \Psi(U_0,\tilde{\mu})\in\mathcal{Z}_h=\mathcal{Y}_h.
\eeqn
Then, by applying the normal form Lemma \ref{lemma3}, we find $U_0=V_0+\Phi_{\tilde{\mu}}(V_0),\ V_0\in\mathcal{O}_1\subset\mathcal{E}_0$, which transforms equation (\ref{G22}) into the normal form
\beq\label{v-expansion}
\frac{dV_0}{dt}=LV_0+N_{\tilde{\mu}} V_0+\rho(V_0,\tilde{\mu}),\nonumber
\eeq
We therefore have
\beq\label{u-expansion}
\tilde{U}=V_0+\tilde{\Psi}(V_0,\tilde{\mu}),\ \
\tilde{\Psi}(V_0,\tilde{\mu})=\Phi_{\tilde{\mu}}(V_0)+\Psi(V_0+\Phi_{\tilde{\mu}}(V_0),\tilde{\mu}).
\eeq
Since $V_0(t)\in\mathcal{E}_0$, it is convenient to write $V_0(t)=A(t)P+\olsi{A(t)P}, \ A(t)\in \mathbb{C}$, then equation (\ref{v-expansion}) can be transformed into an “amplitude equation" $\frac{dA}{dt}=i2n_rn_iA+N_{\tilde{\mu}}(A,\bar{A})+\rho(A,\bar{A},\tilde{\mu})$, which can be simplified further by Lemma \ref{lemma4} to be
\beq\label{A-expansion}
\frac{dA}{dt}=i2n_rn_iA+AQ(|A|^2,\tilde{\mu})+\rho(A,\bar{A},\tilde{\mu}).
\eeq
Here, $Q(|A|^2,\tilde{\mu})=a\tilde{\mu}+b|A|^2+O((|\tilde{\mu}|+|A|^2)^2)$. Hence, equation (\ref{A-expansion}) equals to
\beqn
\frac{dA}{dt}=i2n_rn_iA+a\tilde{\mu} A+b|A|^2A+O(|A|(|\tilde{\mu}|+|A|^2)^2).
\eeqn
We neglect the higher order terms $\rho$, which has no impact on our final analysis. Then, by introducing polar coordinates  $A=re^{i\phi}$,
we obtain
\beqn
\frac{dr}{dt}+ir\frac{d\phi}{dt}=i2n_rn_ir+rQ(r^2,\tilde{\mu}),
\eeqn
whose real and imaginary parts satisfy
\beq\label{real-ima}
\begin{aligned}
\frac{dr}{dt}&=a_r\tilde{\mu} r+b_rr^3+O(r(|\tilde{\mu}|+r^2)^2),\\
\frac{d\phi}{dt}&=2n_rn_i+a_i\tilde{\mu}+b_ir^2+O((|\tilde{\mu}|+r^2)^2),
\end{aligned}
\eeq
in which the real coefficients $a_r$ and $b_r$ represent the real parts of $a$ and $b$, whereas $a_i$ and $b_i$ represent the imaginary parts of $a$ and $b$. The key property of system (\ref{real-ima}) is that the radial equation for $r$ decouples, which corresponds to a pitchfork bifurcation occurs at $\tilde{\mu}=0$. If the bifurcation here is supercritical, then there exists an oscillating solution near $\Gamma_p$, which is called a breathing pulse. Therefore, figuring out the values of coefficients $a$ and $b$ should be the next step. For this, differentiate (\ref{u-expansion}) with respect to $t$ and replace $\frac{dU}{dt}$ and $\frac{dV_0}{dt}$ respectively. There is
\beq\label{identity}
D_{V_0}\tilde{\Psi}(V_0,\tilde{\mu})LV_0-L\tilde{\Psi}(V_0,\tilde{\mu})+N_{\tilde{\mu}} V_0=Q(V_0,\tilde{\mu}),
\eeq
where
\beqn
Q(V_0,\tilde{\mu})=\Pi_p\left(R(V_0+\tilde{\Psi}(V_0,\tilde{\mu}),\tilde{\mu})-D_{V_0}\tilde{\Psi}(V_0,\tilde{\mu})N_{\tilde{\mu} } V_0\right).
\eeqn
Here $\Pi_p$ represents the linear map that associates to a map of class $\mathcal{C}^p$ the polynomial of degree $p$ in its Taylor expansion. We then write the Taylor expansion of $R$ and $\tilde{\Psi}$ as follows:
\beqn
\begin{aligned}
R(\tilde{U},\tilde{\mu})&=\sum_{1\leq q+l\leq p}R_{ql}(\tilde{U}^{(q)},\tilde{\mu}^{(l)})+o((\|\tilde{U}\|+\|\tilde{\mu}\|)^p),\\
\tilde{\Psi}(V_0,\tilde{\mu})&=\sum_{1\leq q+l\leq p}\tilde{\Psi}_{ql}(V_0^{(q)},\tilde{\mu}^{(l)})+o((\|V_0\|+\|\tilde{\mu}\|)^p),
\end{aligned}
\eeqn
with
\beqn
R_{ql}(\tilde{U}^{(q)},\tilde{\mu}^{(l)})=\tilde{U}^q\tilde{\mu}^lR_{ql},\
\tilde{\Psi}_{ql}(V_0^{(q)},\tilde{\mu}^{(l)})=\tilde{\mu}^l\sum_{q_1+q_2=q}A^{q_1}\bar{A}^{q_2}\Psi_{q_1q_2l}.
\eeqn
By identifying in (\ref{identity}) the terms of order $O(\mu)$, $O(A^2)$, and $O(A\bar{A})$, we obtain
\beqn
\begin{aligned}
-L\Psi_{001}&=R_{01},\\
(i4n_rn_ii-L)\Psi_{200}&=R_{20}(P,P),\\
-L\Psi_{110}&=2 R_{20}(P,\bar{P}).
\end{aligned}
\eeqn
Since $\sigma_{pt}=\{\pm2n_rn_ii\}$, the operators $L$ and $(i4n_rn_ii-L)$ on the left sides are invertible, so that $\Psi_{001}$, $\Psi_{200}$ and $\Psi_{110}$ can be uniquely determined. Similarly, we identify the terms of order $O(\mu A)$ and $O(A^2\bar{A})$ to find
\begin{align}
(L- i2n_rn_i)\Psi_{101}&=aP-R_{11}(P)-2R_{20}(P,\Psi_{001}),\label{right}\\
(L- i2n_rn_i)\Psi_{210}&=bP-2R_{20}(P,\Psi_{110})-2R_{20}(\bar{P},\Psi_{200})-3R_{30}(P,P,\bar{P}).\label{right-1}
\end{align}
We already known that $i2n_rn_i$ is a simple isolated eigenvalue of $L$, the range of $(L- i2n_rn_i)$ is of codimension 1, therefore we can solve the above two equations if and only if the right-hand sides satisfy one solvability condition. According to the Fredholm alternative theorem (P.28 of \cite{K13}), it demands that the right-hand sides of (\ref{right}) and (\ref{right-1}) are orthogonal to the kernel of the adjoint operator $(L^*+ i2n_rn_i)$. It is easily to find that the kernel of the adjoint operator $(L^*+ i2n_rn_i)$ is one-dimensional just as the kernel of $(L- i2n_rn_i)$. Moreover, $L^*$ is real. So we introduce the eigenfunction $P^*$ by $L^*P^*=-i2n_rn_iP^*$, which gives
\beq\label{a-b}
\begin{aligned}
a&=\frac{\langle R_{11}(P)+2R_{20}(P,\Psi_{001}),P^*\rangle}{\langle P, P^*\rangle},\\
b&=\frac{\langle 2R_{20}(P,\Psi_{110})+2R_{20}(\bar{P},\Psi_{200})+3R_{30}(P,P,\bar{P}),P^*\rangle}{\langle P, P^*\rangle}.\nonumber
\end{aligned}
\eeq
According to (\ref{G22-1}), $R_{01}=0,\ R_{11}=-1,\ R_{20}(U,V)=\begin{pmatrix}
                                \frac{\alpha G_{15}}{\varepsilon^2}I(\frac{x}{\varepsilon^2})U_2V_2 \\
                                0
                              \end{pmatrix},\ R_{30}=0$,
then
\beq
\begin{aligned}\label{a,b}\nonumber
a&=\frac{\langle -P, P^*\rangle}{\langle P, P^*\rangle}=-1,\\
b&=\frac{\int_{\mathbb{R}}^{\ }\left(4|P_{2}|^2(\frac{\alpha G_{15}}{\varepsilon^2})^2I(\frac{x}{\varepsilon^2})\overline{P_{2}}(-L)^{-1}_{11}I(\frac{x}{\varepsilon^2})+2(\frac{\alpha G_{15}}{\varepsilon^2})^2I(\frac{x}{\varepsilon^2})\overline{P_{2}}(2i\omega_H-L)^{-1}_{11}I(\frac{x}{\varepsilon^2})P_{2}^2\right)\overline{P_1^*}dx}{\langle P, P^*\rangle},
\end{aligned}
\eeq
where subscript $\ _{11}$ represents the location in this operator matrix  and $P^*$ satisfies $(L^*+i2n_rn_i)P^*=0$.
Similarly, $P^*$ can be solved by using the method as before. And
\beq\label{p*}
P_1^*(x)=\begin{cases}
C_5e^{\sqrt{\hat{\mu}-i 2n_rn_i}x}+\frac{b C_6 e^{\sqrt{\frac{\hat{\mu}-i 2n_rn_i}{D}}x}}{(-1+\frac{1}{D})(\hat{\mu}-i 2n_rn_i)},  & x \in I_s^-,\\
C_5+\frac{b C_6}{(-1+\frac{1}{D})(\hat{\mu}-i 2n_rn_i)}, & x \in I_f,\\
C_5e^{-\sqrt{\hat{\mu}-i 2n_rn_i}x}+\frac{b C_6 e^{-\sqrt{\frac{\hat{\mu}-i 2n_rn_i}{D}}x}}{(-1+\frac{1}{D})(\hat{\mu}-i 2n_rn_i)}, & x \in I_s^+,
\end{cases}
\ P_2^*(x)=\begin{cases}\nonumber
C_6e^{\sqrt{\frac{\hat{\mu}-i 2n_rn_i}{D}}x}, & x \in I_s^-,\\
C_6, & x \in I_f,\\
C_6e^{-\sqrt{\frac{\hat{\mu}-i 2n_rn_i}{D}}x}, & x \in I_s^+.
\end{cases}
\eeq
Here, $C_5$ and $C_6$ are likewise required to satisfy existence constraints of the solution.
Then we divide the calculation of $b$ into six parts as
$
b=\frac{I_1+I_2+I_3}{L_1+L_2+L_3},
$
where
\beq
\begin{aligned}
I_1&=\int_{-\infty}^{-\varepsilon}\left(4|P_{2}|^2(\frac{\alpha G_{15}}{\varepsilon^2})^2I(\frac{x}{\varepsilon^2})\overline{P_{2}}(-L)^{-1}_{11}I(\frac{x}{\varepsilon^2})+2(\frac{\alpha G_{15}}{\varepsilon^2})^2I(\frac{x}{\varepsilon^2})\overline{P_{2}}(i4n_rn_i-L)^{-1}_{11}I(\frac{x}{\varepsilon^2})P_{2}^2\right)\overline{P_1^*}dx;\\
I_2&=\int_{-\varepsilon}^{+\varepsilon}\left(4|P_{2}|^2(\frac{\alpha G_{15}}{\varepsilon^2})^2I(\frac{x}{\varepsilon^2})\overline{P_{2}}(-L)^{-1}_{11}I(\frac{x}{\varepsilon^2})+2(\frac{\alpha G_{15}}{\varepsilon^2})^2I(\frac{x}{\varepsilon^2})\overline{P_{2}}(i4n_rn_i-L)^{-1}_{11}I(\frac{x}{\varepsilon^2})P_{2}^2\right)\overline{P_1^*}dx;\\
I_3&=\int_{+\varepsilon}^{+\infty}\left(4|P_{2}|^2(\frac{\alpha G_{15}}{\varepsilon^2})^2I(\frac{x}{\varepsilon^2})\overline{P_{2}}(-L)^{-1}_{11}I(\frac{x}{\varepsilon^2})+2(\frac{\alpha G_{15}}{\varepsilon^2})^2I(\frac{x}{\varepsilon^2})\overline{P_{2}}(i4n_rn_i-L)^{-1}_{11}I(\frac{x}{\varepsilon^2})P_{2}^2\right)\overline{P_1^*}dx;\nonumber\\
L_1&=\int_{-\infty}^{-\varepsilon}P_{1}(x)\overline{P_1^*}(x)+P_{2}(x)\overline{P_2^*}(x)dx,\\
L_2&=\int_{-\varepsilon}^{+\varepsilon}P_{1}(x)\overline{P_1^*}(x)+P_{2}(x)\overline{P_2^*}(x)dx,\\
L_3&=\int_{+\varepsilon}^{+\infty}P_{1}(x)\overline{P_1^*}(x)+P_{2}(x)\overline{P_2^*}(x)dx.\nonumber
\end{aligned}
\eeq
For $I_1$, the operator $(-L)^{-1}_{11}$ and $(2i\omega_H-L)^{-1}_{11}$ are bounded, the function $P$ and $P^*$ are bounded, $I(\frac{x}{\varepsilon^2})$ is exponentially small. Hence, $I_1=0$ to leading order. Likewise, we also have $I_3=0$ to leading order. For $I_2$, something different happens. $I(\frac{x}{\varepsilon^2})$ is not exponentially small. 
 However, to leading order, $P_{2}=C_4+\frac{b C_3}{(-1+D)(\hat{\mu}+i2n_rn_i)}$ and $\overline{P_1^*}=C_5+\frac{b C_6}{(-1+\frac{1}{D})(\hat{\mu}-i 2n_rn_i)}$, which facilitates our calculation. For $L_2$, $P$ and $P^*$ are constants, so we have $L_2=0$ to leading order. For $L_1$ and $L_3$, since the integration interval is symmetric and the integration function is an even function, we get $L_1=L_3$ to leading order. In summary, $b=\frac{I_2}{2L_1}$.
\begin{theorem}\label{thm2}
  Assume that the real part of $b$ is not equal to 0, then for the system (\ref{orginal-nonlinear-3}) a supercritical (resp. subcritical) Hopf bifurcation occurs at $\mu=\hat{\mu}$ when $b_r<0$ (resp., $b_r>0$). More precisely, for $\mu$ sufficiently close to $\hat{\mu}$:
  \begin{enumerate}
\item[(i)] If $b_r>0$ (resp., $b_r<0$), the differential equation has precisely one equilibrium $U(\mu)$ for $\mu<\hat{\mu}$ (resp., $\mu>\hat{\mu}$). This equilibrium is stable when $b_r<0$ and unstable when $b_r>0$;
\item[(ii)] If $b_r>0$ (resp., $b_r<0$), the differential equation possesses for  $\mu>\hat{\mu}$ (resp., $\mu<\hat{\mu}$) an equilibrium $U(\mu)$ and a unique periodic orbit $U^*(\mu)=O(|\mu-\hat{\mu}|^\frac{1}{2})$ which surrounds this equilibrium. The periodic orbit is stable when $b_r<0$ and unstable when $b_r>0$. Moreover, there exists a stable oscillating solution $\Gamma_{osc}$ near $\Gamma_p$, which is given by
    \beqn
    \Gamma_{osc}=\Gamma_{p}+\sqrt{\bigg|\frac{\mu-\hat{\mu}}{b_r}\bigg|}e^{i2n_rn_it}P_{\hat{\mu}}+c.c.+O(|\mu|).
    \eeqn
\end{enumerate}
\end{theorem}
\begin{proof}
The weakly nonlinear stability of the pinned  pulse solution $\Gamma_p$ near a Hopf bifurcation is determined by
\beq
\begin{aligned}
\frac{dr}{dt}&=-\tilde{\mu} r+b_rr^3+O(r(|\tilde{\mu}|+r^2)^2),\\
\frac{d\phi}{dt}&=2n_rn_i+O(|\tilde{\mu}|+r^2),\nonumber
\end{aligned}
\eeq
where $O$-terms originate from the nonlinearity of $R$. The trivial equilibrium $r=0$ of the first radial equation is stable when $\tilde{\mu} >0$ and unstable when $\tilde{\mu}<0$. Moreover, the radial equation has a nontrivial leading order equilibrium if and only if $r^2=\frac{\tilde{\mu}}{b_r}$ exists a positive solution $r^*=O(|\tilde{\mu}|^\frac{1}{2})$. $r^*$ has opposite stability to $r=0$, i.e., $r^*$ is stable if $\tilde{\mu}<0$, while $r^*$ is unstable when $\tilde{\mu}>0$. Moreover, according to (\ref{u-expansion}), $U=A(t)P+\olsi{A(t)P}+\tilde{\Psi}(A(t)P+\olsi{A(t)P},\tilde{\mu})$, which yields the expansion expression for $\Gamma_{osc}$ as
\beqn
   \Gamma_{osc}=\sqrt{\bigg|\frac{\tilde{\mu}}{b_r}\bigg|}e^{i\omega_Ht}P+c.c.+O(|\tilde{\mu}|).
\eeqn
Back to the original system, i.e, reversing the transformation $\left\{\begin{aligned}
                                                                                        \tilde{U}_1&=U_1-\Gamma_{1,p}\\   \tilde{U}_2&=U_2-\Gamma_{2,p}\\  \tilde{\mu}&=\mu-\hat{\mu}
                                                                                         \end{aligned}
                                                                                           \right. $,
we come to the conclusion.
\end{proof}
\begin{remark}
  Similar to Remark \ref{rem1}, when $D=1$, the expressions need to be modified. Not only this,  $P^*_1$ needs to be rewritten as
\beq
P_1^*(x)=\begin{cases}
C_5e^{\sqrt{\hat{\mu}-i 2n_rn_i}x}+\frac{b C_6}{4(\hat{\mu}-i 2n_rn_i)}e^{\sqrt{\hat{\mu}-i 2n_rn_i}x}(-1+2\sqrt{\hat{\mu}-i 2n_rn_i}x),  & x \in I_s^-,\\
C_5-\frac{b C_6}{4(\hat{\mu}-i 2n_rn_i)}, & x \in I_f,\\
C_5e^{-\sqrt{\hat{\mu}-i 2n_rn_i}x}+\frac{b C_6}{4(\hat{\mu}-i 2n_rn_i)}e^{-\sqrt{\hat{\mu}-i 2n_rn_i}x}(-1-2\sqrt{\hat{\mu}-i 2n_rn_i}x), & x \in I_s^+,
\end{cases}
\eeq
which leads to the change of all calculations.
\end{remark}
\subsubsection{$G_2=G_{23}U_2$}
 When $G_2$ is simple linear in $U_2$, i.e., $G_2=G_{23}U_2$, the system admits the pinned pulse solution, provided that
 \begin{flalign*}
&(A6)\   \Delta=4\mu-4\alpha^2G_{11}G_{15}\left(\frac{2b}{(-1+D)\sqrt{\mu}}\frac{\sqrt{D}-D}{2\sqrt{\mu D}-\beta G_{23}}\right)^2>0.&
\end{flalign*}
Correspondingly, we can obtain a univariate cubic equation about $\sqrt{\mu+\lambda}$ analogous to (\ref{cubic}). In order to ensure the existence of Hopf bifurcation, we impose the following assumptions:
\begin{flalign*}
& (A7)\ \Delta>0; &\\
& (A8)\ \beta G_{23}-\sqrt[3]{-(\beta G_{23})^3+3\sqrt{D}(-B+\sqrt{\Delta})}-\sqrt[3]{-(\beta G_{23})^3+3\sqrt{D}(-B-\sqrt{\Delta})}\leq0;\\
& (A9)\ 36D \mu =-\frac{1}{2}\sqrt[3]{\left(-(\beta G_{23})^3+3\sqrt{D}(-B+\sqrt{\Delta})\right)^2}
-\frac{1}{2}\sqrt[3]{\left(-(\beta G_{23})^3+3\sqrt{D}(-B-\sqrt{\Delta})\right)^2}\\
&+(\beta G_{23})^2+\beta G_{23}\left(\sqrt[3]{-(\beta G_{23})^3+3\sqrt{D}(-B+\sqrt{\Delta})}+\sqrt[3]{-(\beta G_{23})^3+3\sqrt{D}(-B-\sqrt{\Delta})}\right)\\
&+2\sqrt[3]{-(\beta G_{23})^3+3\sqrt{D}(-B+\sqrt{\Delta})}\sqrt[3]{-(\beta G_{23})^3+3\sqrt{D}(-B-\sqrt{\Delta})}\  admits\  positive\  roots\  \mu \ with\\
& \Delta=\left(\frac{-2\alpha b \sqrt{D}G_{15}(C_2+\frac{bC_1}{(-1+D)\mu})}{1+\sqrt{D}}\right)\left(324D\left(\frac{-2\alpha b \sqrt{D}G_{15}(C_2+\frac{bC_1}{(-1+D)\mu})}{1+\sqrt{D}}\right)+12(\beta G_{23})^3\right).
\end{flalign*}
\begin{theorem}\label{thm4}
  Assume that (A6), (A7), (A8) and (A9) hold, then the system considered here exists Hopf bifurcation when $\mu=\hat{\mu}$.
\end{theorem}
Now, there are
\beq
\begin{aligned}
\ \ R(\tilde{U},\tilde{\mu})=
\begin{pmatrix}
-\tilde{\mu} \tilde{U}_1+\frac{\alpha G_{15}}{\varepsilon^2}I(\frac{x}{\varepsilon^2})\tilde{U}_2^2 \\
-\tilde{\mu} \tilde{U}_2
\end{pmatrix},
\ L=
\begin{pmatrix}
\frac{d^2}{dx^2}-\hat{\mu} & \frac{2\alpha G_{15}}{\varepsilon^2}I(\frac{x}{\varepsilon^2})\Gamma_{2,p} \\
-b & D\frac{d^2}{dx^2}-\hat{\mu}+\frac{\beta G_{23}}{\varepsilon^2}I(\frac{x}{\varepsilon^2})  \\
\end{pmatrix}.
\end{aligned}\nonumber
\eeq
Hence, $a$ still equals to -1 and the calculation of $b$ can be simplified to the same expression $b=\frac{I_2}{2L_1}$.
\begin{theorem}\label{thm5}
  Assume that the real part of $b$ is not equal to 0, then we get the same conclusions as Theorem \ref{thm2} with $\hat{\mu}$ satisfies Theorem \ref{thm4}.
\end{theorem}
Besides, $\sigma_+=\varnothing$ gives the following remark according to Lemma \ref{lemma2}:
\begin{remark}\label{thm1-1}
  The local center manifold $\mathcal{M}_0(\mu)=V_0+\tilde{\Psi}(V_0,\mu)$ considered in this subsection is locally attracting.
\end{remark}
\begin{remark}\label{rm6}
 Remark \ref{th1}, Theorem \ref{thm2} and \ref{thm5} deal with three simple situations, which involve only one nonlinear term and the simplest linear terms. For other nonlinear $G_1$ and linear $G_2$, we can also ascertain the conditions for the occurrence of breathing pulses by using the same analytical approach.
\end{remark}
\subsection{$G_2$ is Nonlinear and $G_1$ is Linear}\label{34}
Then, we consider the case when $G_2$ is nonlinear and $G_1$ is linear. Since we are interested in mutually coupled systems, we may wish  $G_1=G_{11}+G_{13}U_2$ and $G_2=G_{21}+G_{22}U_1+G_{23}U_2+G_{24}U_1^2+G_{25}U_2^2+G_{26}U_1U_2+...$. Similarly, we study three simple situations.
\subsubsection{$G_2=G_{24}U_1^2$}
First, we assume $G_2$ is a function about $U_1^2$, i.e., $G_2=G_{24}U_1^2$. Then the pinned pulse solution can be uniquely determined by the existence of non-degenerate solutions, i.e., the assumption
\begin{flalign*}
& (A10) \ \sqrt{\frac{D}{\mu}}\left(2\mu+\frac{\alpha G_{13} b}{(\sqrt{D}+1)\sqrt{\mu}}\right)^2>2\alpha^2\beta G_{11}G_{13}G_{24}&
\end{flalign*}
holds. Moreover, we propose  hypothesises:
\begin{flalign*}
& (A11)\ 27 D\sqrt{D}b^2>2\alpha G_{13}(\beta G_{24}C_1)^3(1+\sqrt{D})^2; & \\
& (A12)\ \alpha bG_{13}>0;\\
& (A13)\ 72D\mu=-\sqrt[3]{\left(-3\sqrt{D}\left(\frac{18\alpha Db G_{13}}{1+\sqrt{D}}+\sqrt{\Delta}\right)\right)^2}-\sqrt[3]{\left(-3\sqrt{D}\left(\frac{18\alpha Db G_{13}}{1+\sqrt{D}}-\sqrt{\Delta}\right)\right)^2}\\
&\ +4\sqrt[3]{-3\sqrt{D}\left(\frac{18\alpha Db G_{13}}{1+\sqrt{D}}+\sqrt{\Delta}\right)}\sqrt[3]{-3\sqrt{D}\left(\frac{18\alpha Db G_{13}}{1+\sqrt{D}}-\sqrt{\Delta}\right)} \ admits\  positive\  roots\ \mu,  \ \ \ \ \\
&\ where\  \Delta=\left(\frac{18\alpha Db G_{13} }{1+\sqrt{D}}\right)^2-24\sqrt{D}\left(\alpha\beta G_{13} G_{24} C_1\right)^3.
\end{flalign*}
Therefore, there is
\begin{theorem}
  Assume that (A10), (A11), (A12) and (A13) hold, then the system considered above exists parameter $\hat{\mu}$ such that Hopf bifurcation takes place.
\end{theorem}
Here,
\beq
\begin{aligned}
R(\tilde{U},\tilde{\mu})=
\begin{pmatrix}
-\tilde{\mu} \tilde{U}_1 \\
-\tilde{\mu} \tilde{U}_2+\frac{\beta G_{24}}{\varepsilon^2}I(\frac{x}{\varepsilon^2})\tilde{U}_1^2
\end{pmatrix},
\ L=
\begin{pmatrix}
\frac{d^2}{dx^2}-\hat{\mu} & \frac{\alpha G_{13}}{\varepsilon^2}I(\frac{x}{\varepsilon^2})\\
-b+\frac{2\beta G_{24}}{\varepsilon^2}I(\frac{x}{\varepsilon^2})\Gamma_{1,p} & D\frac{d^2}{dx^2}-\hat{\mu}  \nonumber
\end{pmatrix}.
\end{aligned}
\eeq
Hence, we derive
\beq
\begin{aligned}
a&=-1,\\
b&=\frac{\int_{\mathbb{R}}^{\ }\left(4|P_{1}|^2(\frac{\beta G_{24}}{\varepsilon^2})^2I(\frac{x}{\varepsilon^2})\overline{P_{1}}(-L)^{-1}_{22}I(\frac{x}{\varepsilon^2})+2(\frac{\beta G_{24}}{\varepsilon^2})^2I(\frac{x}{\varepsilon^2})\overline{P_{1}}(i4n_rn_i-L)^{-1}_{22}I(\frac{x}{\varepsilon^2})P_{1}^2\right)\overline{P_2^*}dx}{\langle P, P^*\rangle}.\nonumber
\end{aligned}
\eeq
Now, the operator $(-L)^{-1}_{22}$ and $(i4n_rn_i-L)^{-1}_{22}$ are bounded, and we can obtain $b=\frac{I_2}{2L_1}$ by the similar analysis as before.
\begin{theorem}\label{thm8}
  Assume that the real part of above $b$ is not equal to 0, then for the system above, we get the same conclusions as in Theorem \ref{thm2}.
\end{theorem}
\subsubsection{$G_2=G_{25}U_2^2$}
Then, we assume $G_2$ is a function about $U_2^2$, i.e., $G_2=G_{25}U_2^2$. For this system, its pinned pulse solution must satisfy
\begin{flalign*}
& (A14)\ \left(\frac{4\sqrt{\mu}G_{11}}{\alpha}+\frac{4\sqrt{D}G_{13}\mu}{\alpha \beta G_{25}}+\frac{2b(G_{13})^2(D-\sqrt{D})}{\beta \sqrt{\mu}(-1+D)G_{25}}\right)^2>\frac{16\mu}{\alpha^2}\left((G_{11})^2+\frac{2\sqrt{D\mu}G_{11}G_{13}}{\beta G_{25}}\right).&
\end{flalign*}
Also, we propose hypothesises:
\begin{flalign*}
& (A15)\  27 D\left(\frac{b(D-\sqrt{D})}{-1+D}\right)^2>\frac{8 b(D-\sqrt{D})}{\alpha G_{13}(-1+D)}\left(\beta G_{25}\left(C_2+\frac{b C_1}{(-1+D)\mu}\right)\right)^3;&\\
& (A16)\  2\beta G_{25}\left(C_2+\frac{b C_1}{(-1+D)\mu}\right)\leq\\
&\ \alpha G_{13}\left(\sqrt[3]{-\left(\frac{2\beta G_{25}}{\alpha G_{13}}(C_2+\frac{b C_1}{(-1+D)\mu})\right)^3+\frac{3\sqrt{D}}{\alpha G_{13}}\left(\frac{18bD}{\alpha G_{13}(\sqrt{D}+1)}+\sqrt{\Delta}\right)}\right.\ \ \ \ \ \ \ \ \ \ \ \\
&\ \left.+\sqrt[3]{-\left(\frac{2\beta G_{25}}{\alpha G_{13}}(C_2+\frac{b C_1}{(-1+D)\mu})\right)^3+\frac{3\sqrt{D}}{\alpha G_{13}}\left(\frac{18bD}{\alpha G_{13}(\sqrt{D}+1)}-\sqrt{\Delta}\right)}\right),\\
&\ where\  \Delta=\left(\frac{18bD}{\alpha G_{13}(\sqrt{D}+1)}\right)^2-\frac{96 b\sqrt{D}}{\sqrt{D}+1}\left(\frac{\beta G_{25}}{\alpha G_{13}}(C_2+\frac{b C_1}{(-1+D)\mu})\right)^3;\\
& (A17)\ \frac{\alpha^2(G_{13})^2}{36D}\mu=-\frac{1}{2}\sqrt[3]{\left(-\left(\frac{2\beta G_{25}}{\alpha G_{13}}(C_2+\frac{b C_1}{(-1+D)\mu})\right)^3+\frac{3\sqrt{D}}{\alpha G_{13}}\left(\frac{18bD}{\alpha G_{13}(\sqrt{D}+1)}+\sqrt{\Delta}\right)\right)^2}\\
&\ -\frac{1}{2}\sqrt[3]{\left(-\left(\frac{2\beta G_{25}}{\alpha G_{13}}(C_2+\frac{b C_1}{(-1+D)\mu})\right)^3+\frac{3\sqrt{D}}{\alpha G_{13}}\left(\frac{18bD}{\alpha G_{13}(\sqrt{D}+1)}-\sqrt{\Delta}\right)\right)^2}\\
&\ +\sqrt{-\left(\frac{2\beta G_{25}}{\alpha G_{13}}(C_2+\frac{b C_1}{(-1+D)\mu})\right)^3+\frac{3\sqrt{D}}{\alpha G_{13}}\left(\frac{18bD}{\alpha G_{13}(\sqrt{D}+1)}+\sqrt{\Delta}\right)}\\
&\ \cdot\sqrt{-\left(\frac{2\beta G_{25}}{\alpha G_{13}}(C_2+\frac{b C_1}{(-1+D)\mu})\right)^3+\frac{3\sqrt{D}}{\alpha G_{13}}\left(\frac{18bD}{\alpha G_{13}(\sqrt{D}+1)}-\sqrt{\Delta}\right)}\\
&\ \cdot\left(2+2\left(\frac{\beta G_{25}}{\alpha G_{13}}(C_2+\frac{b C_1}{(-1+D)\mu})\right)\right)+4\left(\frac{\beta G_{25}}{\alpha G_{13}}(C_2+\frac{b C_1}{(-1+D)\mu})\right)^2 \ admits\  positive\ roots\  \mu.
\end{flalign*}
Hence, we derive
\begin{theorem}
  Assume that (A14), (A15), (A16) and (A17) hold, then the Hopf bifurcation takes place with the bifurcation parameter $\mu=\hat{\mu}$.
\end{theorem}
In this case, there are
\beq
\begin{aligned}
\ \ R(\tilde{U},\tilde{\mu})=
\begin{pmatrix}
-\tilde{\mu} \tilde{U}_1 \\
-\tilde{\mu} \tilde{U}_2+\frac{\beta G_{25}}{\varepsilon^2}I(\frac{x}{\varepsilon^2})\tilde{U}_2^2
\end{pmatrix},
\ L=
\begin{pmatrix}
\frac{d^2}{dx^2}-\hat{\mu} & \frac{\alpha G_{13}}{\varepsilon^2}I(\frac{x}{\varepsilon^2})\\
-b & D\frac{d^2}{dx^2}-\hat{\mu}+\frac{2\beta G_{25}}{\varepsilon^2}I(\frac{x}{\varepsilon^2})\Gamma_{2,p}  \nonumber
\end{pmatrix}
\end{aligned},
\eeq
which give
\beq
\begin{aligned}
a&=-1,\\
b&=\frac{\int_{\mathbb{R}}^{\ }\left(4|P_{2}|^2(\frac{\beta G_{24}}{\varepsilon^2})^2I(\frac{x}{\varepsilon^2})\overline{P_{2}}(-L)^{-1}_{22}I(\frac{x}{\varepsilon^2})+2(\frac{\beta G_{24}}{\varepsilon^2})^2I(\frac{x}{\varepsilon^2})\overline{P_{2}}(i4n_rn_i-L)^{-1}_{22}I(\frac{x}{\varepsilon^2})P_{2}^2\right)\overline{P_2^*}dx}{\langle P, P^*\rangle}.
\end{aligned}\nonumber
\eeq
\begin{theorem}\label{thm10}
  Assume that the real part of above $b$ is not equal to 0, then for this case, we get the same conclusions as Theorem \ref{thm2}.
\end{theorem}
\subsubsection{$G_2=G_{26}U_1U_2$}
Finally, we consider the case when $G_2=G_{26}U_1U_2$, whose pulse solution can be given if
\beali\label{more}
2\sqrt{\mu}C_1&=\alpha\left(G_{11} +G_{13}\left(C_2+\frac{b C_1}{(-1+D)\mu}\right)\right),\\
2\left(\sqrt{\frac{\mu}{D}}C_2+\frac{b\sqrt{\mu}}{(-1+D)\mu}C_1\right)&=\frac{\beta G_{26}C_1}{D}\left(C_2+\frac{b C_1}{(-1+D)\mu}\right)
\enali
admits non-degenerate solutions $C_1$ and $C_2$. Correspondingly, the existence condition of eigenvalue function is that
\begin{small}
\beali\label{and more}
2\sqrt{\mu+\lambda}C_3&=\alpha\left(G_{13}\left(C_4+\frac{b C_3}{(-1+D)(\mu+\lambda)}\right)\right),\\
\frac{\beta G_{26}C_3}{D}\left(C_2+\frac{b C_1}{(-1+D)\mu}\right)&+\frac{\beta G_{26}C_1}{D}\left(C_4+\frac{b C_3}{(-1+D)(\mu+\lambda)}\right)=2\sqrt{\frac{\mu+\lambda}{D}}C_4+\frac{2b C_3}{(-1+D)\sqrt{\mu+\lambda}}
\enali
\end{small}
admits non-degenerate solutions. Assuming that the existence assumptions about (\ref{more}) and (\ref{and more}) hold, then there exists a parameter $\hat{\mu}$ such that this system undergoes a Hopf bifurcation. We omit the exact expression here due to its verbosity.

Note that in this case
\beq
R(\tilde{U},\tilde{\mu})=
\begin{pmatrix}
-\tilde{\mu} \tilde{U}_1 \\
-\tilde{\mu} \tilde{U}_2+\frac{\beta G_{26}}{\varepsilon^2}I(\frac{x}{\varepsilon^2})\tilde{U}_1\tilde{U}_2
\end{pmatrix},
\
L=
\begin{pmatrix}
\frac{d^2}{dx^2}-\hat{\mu} & \frac{\alpha G_{13}}{\varepsilon^2}I(\frac{x}{\varepsilon^2})\\
-b+\frac{\beta G_{26}}{\varepsilon^2}I(\frac{x}{\varepsilon^2})\Gamma_{2,p} & D\frac{d^2}{dx^2}-\hat{\mu}+\frac{\beta G_{26}}{\varepsilon^2}I(\frac{x}{\varepsilon^2})\Gamma_{1,p}  \\
\end{pmatrix},\nonumber
\eeq
i.e.,
\beq
R_{01}=0,\ R_{11}=-1,\ R_{20}(U,V)=\binom{0}{\frac{\beta G_{26}}{2\varepsilon^2}I(\frac{x}{\varepsilon^2})(U_1V_2+V_1U_2) },\ R_{30}=0.\nonumber
\eeq
In this case, we can still calculate $a$ and $b$. However, the calculation of $b$ requires many lengthy equalities so we will not delve into more detail here.
\begin{remark}
  Theorem \ref{thm8} and \ref{thm10} study the situations where $G_2$ exhibits the simplest nonlinearity. When $G_2$ involves more nonlinear terms, we can still analyze in a similar way.
\end{remark}
Also, $\sigma_+=\varnothing$ gives the same remark:
\begin{remark}
  The local center manifold $\mathcal{M}_0(\mu)=V_0+\tilde{\Psi}(V_0,\mu)$ considered in this subsection is locally attracting.
\end{remark}
\begin{remark}\label{412}
  In the subsection \ref{33} and \ref{34}, we use the assumption that the real part of $b$ is nonzero, which can be removed if we consider the higher order terms like $O(A|A|^4)$ etc.
\end{remark}
\subsection{A concrete example}\label{35}
To illustrate the above Remark \ref{rm6}, we present the computation of a specific example, which was previously proposed in \cite{DvHS} to analyze the Hopf bifurcation. The system we are considering is
\beali\label{nu}
  \frac{\partial U_1}{\partial t}&=\frac{\partial^2 U_1}{\partial x^2}-\mu U_1+\frac{2}{\varepsilon^2}I(\frac{x}{\varepsilon^2})(U_2+1+\nu U_1^3),\\
  \frac{\partial U_2}{\partial t}&=4\frac{\partial^2 U_2}{\partial x^2}-\frac{\sqrt{3}}{3} U_1-\mu U_2+\frac{2}{\varepsilon^2}I(\frac{x}{\varepsilon^2})(U_1+2).
\enali
For this system, if polynomial equations
\beali\label{nu1}
\sqrt{\mu}\tilde{C}_1&=\tilde{C_2}+\frac{\tilde{C_1}}{3\sqrt{3}\mu}+1+\nu \tilde{C_1}^3,\\
\tilde{C_1}+2&=4\left(\frac{\sqrt{\mu}}{2}\tilde{C_2}+\frac{\tilde{C_1}}{3\sqrt{3\mu}}\right)
\enali
admit non-degenerate solutions $\tilde{C_1}$ and $\tilde{C_2}$. Its pinned one-pulse solution, to leading order, can be given as
\beq
\tilde{U}_{1,p}(x)=\begin{cases}
\tilde{C}_1e^{\sqrt{\mu} x}, & x \in I_s^-,\\
\tilde{C}_1,& x \in I_f,\\
\tilde{C}_1e^{-\sqrt{\mu}x}, & x \in I_s^+,\nonumber
\end{cases}
\ \tilde{U}_{2,p}(x)=\begin{cases}
\tilde{C}_2e^\frac{\sqrt{\mu}x}{2}+\frac{ \tilde{C_1}}{3\sqrt{3}\mu}e^{\sqrt{\mu}x}, & x \in I_s^-,\\
\tilde{C_2}+\frac{ \tilde{C_1}}{3\sqrt{3}\mu}, & x \in I_f,\\
\tilde{C}_2e^\frac{-\sqrt{\mu}x}{2}+\frac{ \tilde{C_1}}{3\sqrt{3}\mu}e^{-\sqrt{\mu}x}, & x \in I_s^+.\nonumber
\end{cases}
\eeq
By calculation, equations (\ref{nu1}) equal to
\beq\label{nu2}
\nu\sqrt{\mu}\tilde{C_1}^3+\frac{3\sqrt{3\mu}-2-6\sqrt{3\mu}\mu}{6\sqrt{3\mu}}\tilde{C_1}+1+\sqrt{\mu}=0.
\eeq
If we consider the small nonlinearity as perturbations to the linear $G_1$, we may as well assume $\nu\ll1$. Then (\ref{nu2}) can be solved as: $\tilde{C_1}=\frac{-6\sqrt{3}\mu-6\sqrt{3\mu}}{-6\sqrt{3\mu}\mu+3\sqrt{3\mu}-2}$, $\tilde{C_2}=\frac{6\sqrt{3\mu}\mu+3\sqrt{3}\mu-4\sqrt{\mu}-2}{6\sqrt{3}\mu^2-3\sqrt{3}\mu+2\sqrt{\mu}}$, to leading order. The eigenvalue problem about this pinned 1-pulse solution $\tilde{\Gamma}_p(x)=(\tilde{\Gamma}_{1,p},\tilde{\Gamma}_{2,p})^T=(\tilde{U}_{1,p}(x)+O(\varepsilon), \tilde{U}_{2,p}(x)+O(\varepsilon))^T$ is
\beali
  0&=\frac{d^2 p_1}{d x^2}-(\mu+\lambda) p_1+\frac{2}{\varepsilon^2}I(\frac{x}{\varepsilon^2})(p_2+3\nu\tilde{\Gamma}_{1,p}^2p_1),\\
  0&=4\frac{d^2 p_2}{d x^2}-\frac{\sqrt{3}}{3} p_1-(\mu+\lambda) p_2+\frac{2}{\varepsilon^2}I(\frac{x}{\varepsilon^2})(p_1),
\enali
and its eigenfunction
\beq
\tilde{P_1}(x)=\begin{cases}
\tilde{C_3}e^{\sqrt{\mu+\lambda}x}, & x \in I_s^-,\\
\tilde{C_3}, & x \in I_f,\\
\tilde{C_3}e^{-\sqrt{\mu+\lambda}x}, & x \in I_s^+,\nonumber
\end{cases}
\ \tilde{P_2}(x)=\begin{cases}
\tilde{C}_4e^\frac{\sqrt{\mu+\lambda}x}{2}+\frac{ \tilde{C_3}}{3\sqrt{3}(\mu+\lambda)}e^{\sqrt{\mu+\lambda}x}, & x \in I_s^-,\\
\tilde{C}_4+\frac{ \tilde{C_3}}{3\sqrt{3}(\mu+\lambda)}, & x \in I_f,\\
\tilde{C}_4e^\frac{-\sqrt{\mu+\lambda}x}{2}+\frac{ \tilde{C_3}}{3\sqrt{3}(\mu+\lambda)}e^{-\sqrt{\mu+\lambda}x}, & x \in I_s^+,
\end{cases}\nonumber
\eeq
must satisfy
\beali\label{nu3}
\sqrt{\mu+\lambda}\tilde{C_3}&=\tilde{C_4}+\frac{\tilde{C_3}}{3\sqrt{3}(\mu+\lambda)}+3\nu\tilde{C_1}^2\tilde{C_3},\\
\tilde{C_3}&= 4\left(\frac{\sqrt{\mu+\lambda}\tilde{C_4}}{2}+\frac{\tilde{C_3}}{3\sqrt{3}\sqrt{\mu+\lambda}}\right).
\enali
To leading order, the equation (\ref{nu3}) is equivalent to $(\sqrt{\mu+\lambda})^3-\frac{\sqrt{\mu+\lambda}}{2}+\frac{1}{3\sqrt{3}}=0$.
Its solutions are
\beq
\begin{aligned}
t_{1,2}&=\frac{(1\pm i\sqrt{3})(\frac{1}{2}(-\sqrt{\frac{3}{2}}+\sqrt{3}))^{\frac{1}{3}}}{2 \cdot 3^{\frac{2}{3}}}+\frac{(1 \mp i \sqrt{3})}{2 \cdot 2^{\frac{2}{3}}(3(-\sqrt{\frac{3}{2}}+\sqrt{3}))^{\frac{1}{3}}},\\
t_3&=-(\frac{1}{2}(-\sqrt{\frac{3}{2}}+\sqrt{3}))^{\frac{1}{3}} \cdot \frac{1}{3^{\frac{2}{3}}}-\frac{1}{2^{\frac{2}{3}}(3 (-\sqrt{\frac{3}{2}}+\sqrt{3}))^{\frac{1}{3}}}.\nonumber
\end{aligned}
\eeq
Therefore, when $\hat{\mu}=\frac{1}{12}(4-\sqrt[3]{3+2\sqrt{2}}-\sqrt[3]{3-2\sqrt{2}})$, Hopf bifurcation takes place and $\sigma_{pt}= \sigma_0:=\{\pm i\omega_H\}=\{\pm\frac{\sqrt{3}}{12}(\sqrt[3]{3+2\sqrt{2}}-\sqrt[3]{3-2\sqrt{2}})i\}$. We denote the eigenfunction at the Hopf bifurcation as $(\hat{P}_{1},\hat{P}_{2})^T$.
Then, in order to unfold this bifurcation, we transform system (\ref{nu}) to
\beq
\begin{aligned}
\frac{d\begin{pmatrix}
\tilde{U}_1\\
\tilde{U}_2
\end{pmatrix}}{dt}&=
\begin{pmatrix}
\frac{d^2}{dx^2}-\hat{\mu}+\frac{2}{\varepsilon^2}I(\frac{x}{\varepsilon^2})(3\nu \tilde{\Gamma}_{1,p}^2) & \frac{2}{\varepsilon^2}I(\frac{x}{\varepsilon^2})\\
-\frac{\sqrt{3}}{3}+\frac{2}{\varepsilon^2}I(\frac{x}{\varepsilon^2}) & 4\frac{d^2}{dx^2}-\hat{\mu}
\end{pmatrix}
\begin{pmatrix}
\tilde{U}_1\\
\tilde{U}_2
\end{pmatrix}
+
\begin{pmatrix}
  -\tilde{\mu}\tilde{U}_1+\frac{2}{\varepsilon^2}I(\frac{x}{\varepsilon^2})(\nu\tilde{U}_1^3+3\nu\tilde{U}_1^2\tilde{\Gamma}_{1,p})\\
  -\tilde{\mu}\tilde{U}_2
\end{pmatrix},
\end{aligned}
\eeq
whose
$R_{01}=0,\ R_{11}=-1,\ R_{20}(U,V)=\binom{\frac{6\nu}{\varepsilon^2}I(\frac{x}{\varepsilon^2})\tilde{\Gamma}_{1,p}U_1V_1 }{0},\ R_{30}(U,V,W)=\binom{\frac{2\nu}{\varepsilon^2}I(\frac{x}{\varepsilon^2})U_1V_1W_1 }{0}$.
Therefore, to leading order, we have
\beqn
-L\Psi_{001}=0,\ \ \ \ \ \ \ \ \ \ \ \ \ \
\eeqn
\beqn
\begin{aligned}
(2i\omega_H-L)\Psi_{200}&=\binom{\frac{6\nu}{\varepsilon^2}I(\frac{x}{\varepsilon^2})\tilde{U}_{1,p}\hat{P}_1^2 }{0},\\
-L\Psi_{110}&=\binom{\frac{12\nu}{\varepsilon^2}I(\frac{x}{\varepsilon^2})\tilde{U}_{1,p}\hat{P}_1\overline{\hat{P}}_1} {0}.
\end{aligned}
\eeqn
Owing to the fact that $2i\omega_H$ and 0 do not belong to $\sigma(L)$, we can derive the leading order expressions for $\Psi_{200}$ and $\Psi_{110}$ as
\beq
\Psi_{200,1}(x)=\begin{cases}
\tilde{C_7}e^{\sqrt{\hat{\mu}+2i\omega_H}x}, & x \in I_s^-,\\
\tilde{C_7}, & x \in I_f,\\
\tilde{C_7}e^{-\sqrt{\hat{\mu}+2i\omega_H}x}, & x \in I_s^+,\nonumber
\end{cases}
\ \Psi_{200,2}(x)=\begin{cases}
\tilde{C}_8e^\frac{\sqrt{\hat{\mu}}x}{2}+\frac{ \sqrt{3}\tilde{C_7}e^{\sqrt{\hat{\mu}+2i\omega_H}x}}{3(3\hat{\mu}+8i\omega_H)}, & x \in I_s^-,\\
\tilde{C}_8+\frac{ \sqrt{3}\tilde{C_7}}{3(3\hat{\mu}+8i\omega_H)}, & x \in I_f,\\
\tilde{C}_8e^\frac{-\sqrt{\hat{\mu}}x}{2}+\frac{ \sqrt{3}\tilde{C_7}e^{-\sqrt{\hat{\mu}+2i\omega_H}x}}{3(3\hat{\mu}+8i\omega_H)}, & x \in I_s^+,
\end{cases}\nonumber
\eeq
\beq
\Psi_{110,1}(x)=\begin{cases}
\tilde{C_9}e^{\sqrt{\hat{\mu}}x}, & x \in I_s^-,\\
\tilde{C_9}, & x \in I_f,\\
\tilde{C_9}e^{-\sqrt{\hat{\mu}}x}, & x \in I_s^+,\nonumber
\end{cases}
\ \ \ \Psi_{110,2}(x)=\begin{cases}
\tilde{C}_{10}e^\frac{\sqrt{\hat{\mu}}x}{2}+\frac{ \sqrt{3}\tilde{C_7}}{9\hat{\mu}}e^{\sqrt{\hat{\mu}}x}, & x \in I_s^-,\\
\tilde{C}_{10}+\frac{ \sqrt{3}\tilde{C_7}}{9\hat{\mu}}, & x \in I_f,\\
\tilde{C}_{10}e^\frac{-\sqrt{\hat{\mu}}x}{2}+\frac{ \sqrt{3}\tilde{C_7}}{9\hat{\mu}}e^{-\sqrt{\hat{\mu}}x}, & x \in I_s^+,
\end{cases}\nonumber
\eeq
with $\tilde{C}_7$ and $\tilde{C}_8$ satisfy
\beali
2\sqrt{\hat{\mu}+2i\omega_H}\tilde{C}_7&=6\nu\tilde{C}_1^2\tilde{C}_7+2\left(+\frac{\sqrt{3}\tilde{C}_7}{3(3\hat{\mu}+8i\omega_H)}\right)+6\nu\tilde{C}_1|\tilde{C}_3|^2,\\
\frac{1}{2}\tilde{C}_7&= 2\left(\frac{\sqrt{\hat{\mu}}\tilde{C}_8}{2}+\frac{\sqrt{3}\tilde{C}_7\sqrt{\hat{\mu}+2i\omega_H}}{3(3\hat{\mu}+8i\omega_H)}\right),
\enali
$\tilde{C}_9$ and $\tilde{C}_{10}$ satisfy
\beali
2\sqrt{\hat{\mu}}\tilde{C}_9&=6\nu\tilde{C}_1^2\tilde{C}_9+2\left(\tilde{C}_{10}+\frac{\sqrt{3}\tilde{C}_9}{9\hat{\mu}}\right)+12\nu\tilde{C}_1|\tilde{C}_3|^2,\\
\frac{1}{2}\tilde{C}_9&= 2\left(\frac{\sqrt{\hat{\mu}}\tilde{C}_{10}}{2}+\frac{\sqrt{3}\tilde{C}_9}{9\sqrt{\hat{\mu}}}\right).
\enali
Therefore, we can derive the expression of $b$ as
$
b=\frac{\hat{I}_{2,1}+\hat{I}_{2,2}+\hat{I}_{2,3}}{2\hat{L}},
$
 where
\beq
\begin{aligned}
&\hat{I}_{2,1}=\int_{-\varepsilon}^{\varepsilon}\frac{12\nu}{\varepsilon^2}I(\frac{x}{\varepsilon^2})\tilde{C}_1\tilde{C}_3\tilde{C}_9\overline{\left(\tilde{C}_5-\frac{4\sqrt{3}\tilde{C}_6}{9(\hat{\mu}-i\omega_H)}\right)}dx,\\
&\hat{I}_{2,2}=\int_{-\varepsilon}^{\varepsilon}\frac{12\nu}{\varepsilon^2}I(\frac{x}{\varepsilon^2})\tilde{C}_1\overline{\tilde{C}_3}\tilde{C}_7\overline{\left(\tilde{C}_5-\frac{4\sqrt{3}\tilde{C}_6}{9(\hat{\mu}-i\omega_H)}\right)}dx,\\
&\hat{I}_{2,3}=\int_{-\varepsilon}^{\varepsilon}\frac{12\nu}{\varepsilon^2}I(\frac{x}{\varepsilon^2}){\tilde{C}_3}^2\overline{\tilde{C}_3}\overline{\left(\tilde{C}_5-\frac{4\sqrt{3}\tilde{C}_6}{9(\hat{\mu}-i\omega_H)}\right)}dx,\\
&\hat{L}=2\int_{-\infty}^{-\varepsilon}\left(e^{\sqrt{\hat{\mu}+i\omega_H}x}\overline{\tilde{C}_5e^{\sqrt{\hat{\mu}+i\omega_H}x}-\frac{4\sqrt{3}\tilde{C}_6e^{\frac{\sqrt{\hat{\mu}+i\omega_H}x}{2}}}{9(\hat{\mu}-i\omega_H)}}\right.\\
&\ \ \  \left.+\left(\left(\sqrt{\hat{\mu}+i{\omega_H}}-\frac{1}{3\sqrt{3}(\hat{\mu}+i{\omega_H})}\right)e^\frac{\sqrt{\hat{\mu}+i{\omega_H}}x}{2}+\frac{ 1}{3\sqrt{3}(\hat{\mu}+i{\omega_H})}e^{\sqrt{\hat{\mu}+i{\omega_H}}x}\right)\overline{\tilde{C}_6e^{\frac{\sqrt{\hat{\mu}+i\omega_H}x}{2}}}\right)dx.\nonumber
\end{aligned}
\eeq
When $\mu=0.1<\hat{\mu}$, $\nu=-0.001$, $\varepsilon=0.1$, by numerical calculation, $a=1,\ b\approx-1.49318 + 3.7192 i$. By Theorem \ref{thm2}, system (\ref{nu}) admits a stable periodic solution, whose numerical simulation is shown in Figure \ref{stable osc}. This result is consistent with the numerical simulation in Figure 5$(b)$ \cite{DvHS}. If we change $\nu$ to 0.001, then there are $a=1,\ b\approx1.49318 - 3.7192 i$, which leads to an unstable equilibrium like Figure \ref{unstable equ}. If we set bifurcation parameter $\mu$ to be 0.2, i.e., $\mu>\hat{\mu}$, then we have Figure \ref{unstable osc} when $\nu=0.001$ and Figure \ref{stable equ} when $\nu=-0.001$.
\begin{figure}[htbp]
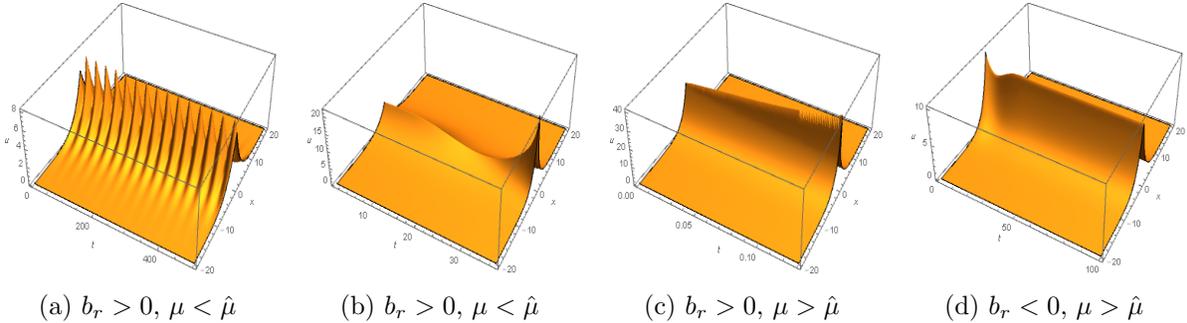

  \centering
  \subcaptionbox{$b_r>0$, $\mu<\hat{\mu}$\label{stable osc}}{
  \includegraphics[width=0.22\textwidth,height=0.22\textwidth]{stableosc.pdf}
  }
  \subcaptionbox{$b_r>0$, $\mu<\hat{\mu}$\label{unstable equ}}{
  \includegraphics[width=0.22\textwidth,height=0.22\textwidth]{unstableequ.pdf}
  }
  \subcaptionbox{$b_r>0$, $\mu>\hat{\mu}$\label{unstable osc}}{
  \includegraphics[width=0.22\textwidth,height=0.22\textwidth]{unstableosc.pdf}
  }
  \subcaptionbox{$b_r<0$, $\mu>\hat{\mu}$\label{stable equ}}{
  \includegraphics[width=0.22\textwidth,height=0.22\textwidth]{stableequ.pdf}
  }
  \caption{numerical simulation of stationary solutions around Hopf bifurcation}
\end{figure}
\section{Discussion}\label{4}
This paper presents a research inspired by numerical simulations in a linear reaction–diffusion system with strong spatially localized impurities \cite{DvHS} and the technical approach of Hopf normal form advocated in \cite{veerman2015breathing}. It demonstrates that introducing a small nonlinearity to the fast variables can stabilize the stationary pulses that would otherwise blow up, and create new stable oscillating pulses when the stationary pulses undergo a Hopf bifurcation, which is exhibited in Remark \ref{rm6} and numerically verified in subsection \ref{35}.  This confirms once again that the Hopf bifurcation could be the birthplace of breathing pulses \cite{firth2002dynamical}. Besides, subsections \ref{33} and \ref{34} present a series of simple nonlinear cases in which the breathing pulses may emerge around the Hopf bifurcation when specific conditions are satisfied.

Furthermore, the center manifold associated with the Hopf bifurcation can be expanded not only up to third order as we already emphasized in Remark \ref{412}. Thus, our next step is to implement the fifth order expansion around the generalized Bautin point. Correspondingly, the non-zero assumption for $b_r$ will be replaced, and the normal form will be changed to
\beqn
\frac{dA}{dt}=i\omega_HA+(a_0\hat{\mu}+a_1\hat{\mu}^2)A+(b_0+b_1\hat{\mu})A|A|^2+cA|A|^4+o(|\hat{\mu}|^3, |A|^5).
\eeqn
It involves more calculations and more inverse problems, while the descriptions for the dynamically modulated pulse amplitudes are available, whose amplitude may be quasiperiodically or even chaotically modulated \cite{VD13}.

Moreover, a detailed unfolding of a Bogdanov–Takens bifurcation \cite{kuznetsov1998elements,argentina1997chaotic} or a Dumortier–Roussarie –Sotomayor bifurcation \cite{dumortier1987generic} of a localized pinned pulse solution is expected to be continued based on our paper and the analysis in \cite{veerman2015breathing}.  In this case, the multiple eigenvalue $\lambda=0$ must be analyzed separately. In turn, these results may serve as a first analytical step towards understanding the pattern formations \cite{coombes2014neural}. As stated earlier, oscillating fronts also play an important role in understanding these phenomena analytically, for example, as shown in \cite{chirilus2019unfolding}. However, whether explicit expressions for pinned fronts (and, more importantly, their eigenfunctions) can be given is an important prerequisite. According to \cite{CLSZ} and \cite{LSZ}, it can be solved easily and the normal form expansion parameters can therefore be evaluated. In particular, the analysis procedure is also valid for multi-pulses in \cite{CLSZ}, which are composed of piecewise periodic solutions with explicit expressions. Especially, the evenness of the eigenfunctions can bring great convenience in calculating the normal form expansions.

\section*{Acknowledgement}
J. Li, Q. Yu and Q. Zhang are supported by NSFC 12171174.

\section*{Data available}
The data that support the findings of this study are available within the article.

\bibliographystyle{plain}

\end{document}